\documentclass{article}

\usepackage{amsmath}
\usepackage{amsfonts}
\usepackage{amsthm}
\usepackage{tikz-cd}
\usepackage{booktabs}
\usepackage{enumerate}
\usepackage{graphicx}
\usepackage{etoolbox}
\usepackage[normalem]{ulem}
\newtoggle{arxiv}
\toggletrue{arxiv}
\newtoggle{colorarxivdifferences}
\toggletrue{colorarxivdifferences}

\newcommand{\Pb}{{\mathbb P}}
\newcommand{\R}{{\mathbb R}}
\newcommand{\E}{{\mathbb E}}
\newcommand{\ed}{{\mathrm d}}

\newcommand{\Vol}{{\hbox{Vol}}}

\newcommand{\bx}{{\bar X}}
\newcommand{\tx}{{\tilde X}}

\newcommand{\Ti}{\mathbb{T}}

\newcommand{\horzvect}[2]{\left( #1, #2  \right)}
\newcommand{\url}[1]{{\tt{ #1 }}}

\newcommand{\arxivhighlight}[1]{\iftoggle{colorarxivdifferences}{{\color{black}{#1}}}{}}
\newcommand{\ifarxiv}[2]{\iftoggle{arxiv}{\arxivhighlight{#1}}{\arxivhighlight{#2}}}

\newtheorem{lemma}{Lemma}
\newtheorem{theorem}{Theorem}
\newtheorem{proposition}{Proposition}

\theoremstyle{definition}
\newtheorem{definition}{Definition}
\newtheorem{example}{Example}

\newtheorem{corollary}{Corollary}

\newcommand{\jet}{\  \begin{tikzpicture}\draw[->]{(0,0) arc (210:330:0.3)}; \end{tikzpicture}\  }
\begin{document}

\title{{\normalsize Updated version``Intrinsic SDEs as jets'' in {\emph{Proceedings of the Royal Society A}}, February 2018, {\small DOI 0.1098/rspa.2017.0559}} \\ \hspace{1cm} \\ 
 {\bf \Large Coordinate-free  Stochastic Differential Equations as Jets}}
\author{John Armstrong \\ Dept. of Mathematics \\ King's College London \\ {\small \tt{john.1.armstrong@kcl.ac.uk}} \and Damiano Brigo 
\\ Dept. of Mathematics \\ Imperial College London \\ {\small \tt{damiano.brigo@imperial.ac.uk}}
}

\date{First version: October 1, 2015. This version: \today}

\maketitle

\begin{abstract}
We explain how It\^o Stochastic Differential Equations (SDEs) on manifolds may be defined using 2-jets of smooth functions. We show how this relationship can be interpreted in terms of a convergent numerical scheme. We show how jets can be used to derive graphical representations of It\^o SDEs. We show how jets
can be used to derive the differential operators associated with SDEs
in a coordinate free manner. We relate jets to vector flows, giving a
geometric interpretation of the It\^o--Stratonovich transformation. 
We show how percentiles can be used to give an alternative
coordinate free interpretation of the coefficients of one dimensional SDEs.
We relate this to the jet approach. This allows us to interpret the coefficients of SDEs in terms of ``fan diagrams''. In particular
the median of a SDE solution is associated to the drift of the SDE in Stratonovich form for small times.
\end{abstract}
%%%%%%%%%%%%%%%%%%%%%%%%%%%

\medskip


\medskip

AMS classification codes: 58A20, 39A50, 58J65,  60H10,  60J60, 65D18  
% 39A50: Stoch Difference Eq
% 58A20   Jets 
% 58J65   Diffusion processes and stochastic analysis on manifolds  
% 60H10 Stochastic ordinary differential equations 
% 60J60  Diffusion processes  
% 65D18   Computer graphics, image analysis, and computational geometry  

%\tableofcontents

\newpage

%%%%%%%%%% Insert the texts which can accomdate on firstpage in the tag "fmtext" %%%%%

%\begin{fmtext}
\section{Introduction}

Stochastic Differential Equations  (SDEs) on manifolds were first defined by It\^o in \cite{ito2}. It\^o's formulation was given in terms of coordinate charts. This has lead many authors to seek
coordinate free formulations of SDEs on manifolds. We will describe such
a formulation in the language of $2$-jets \cite{saunders}. We will study how this
formulation gives rise to intuitive graphical representations of SDEs.

Coordinate free formulations of SDEs have been given previously. One approach is to use Stratonovich calculus (see \cite{elworthy,elworthybook,rogers}). Another is the theory of second order tangent vectors, diffusors and Schwartz morphism (see \cite{emery, emeryhal}). A third is via the It\^o bundle (see \cite{belopolskaja}, \cite{gliklikh} or the appendix in \cite{brzezniak}).

The value of the $2$-jet approach is that it is particularly simple
and intuitive. In particular, as the notion of a jet is already familiar 
to differential geometers, we do not need to introduce novel differential geometric constructs.
%\end{fmtext}

%%%%%%%%%%%%%%% End of first page %%%%%%%%%%%%%%%%%%%%%

\maketitle

In Section \ref{mainSection} we will give an informal description of
the definition of SDEs in the language of $2$-jets. This description
does not require the reader to have prior experience of SDEs (though we do
assume they know the definition of Brownian motion). \ifarxiv{For the reader convenience, in Appendix \ref{appendix:classicalFormulation} we present a quick informal introduction to SDEs in It\^o or Stratonovich form, the two mainstream stochastic calculi, including the way to switch from one calculus to the other one, namely the It\^o-Stratonovich transformation. \\ 
It is well known that vectors on a manifold can be understood from a number of different perspectives.
Firstly tangent vectors can be defined in terms of local coordinates and their transformation rules. Secondly
one can define vectors in terms of operators, specifically as linear functions on the space of germs
of smooth functions that also obey the Leibniz rule. A third definition of a vector is as a first order approximation to a curve on a manifold. Fourth, one can understand a vector field as an infinitesimal diffeomorphism of a manifold. Since in addition vector fields can be interpreted as ordinary differential equations (ODEs) on a manifold, this gives four ways of understanding ODEs.

In this paper we will study the  parallel interpretations for SDEs. The first approach was used by It\^o in \cite{ito2} to define SDEs on manifolds. The second approach is coordinate free and is related to understanding SDEs in terms of diffusion operators (see \cite{friedmansdes} for the case of $\R^n$). The third approach corresponds to our interpretation of SDEs in terms of jets.
The fourth approach corresponds to Stratonovich calculus. Many texts such as \cite{elworthy,elworthybook,rogers}, use Stratonovich calculus to define SDEs on manifolds.\\
}{} 
For simplicity we first consider the case of an SDE driven by a single Brownian motion. Our description of SDEs is given by writing down a system of difference equations using a coordinate free notation. A formal proof that the solutions of these equations converge to the solutions of the classically defined It\^o SDEs is given in Appendix \ref{proofsSection}.

We also consider how SDEs can be understood graphically. In particular we will see how $2$-jets allow us
to draw an SDE in a way
that makes the transformation law of SDEs, known as It\^o's lemma, intuitively clear. We will illustrate a way
of drawing an SDE on a rubber sheet such that if the sheet is stretched,
the diagram transforms according to It\^o's lemma. In other words given
an SDE in $\R^n$ we give a method of drawing SDEs such that for all 
well-behaved $f:\R^n \to \R^n$
the following diagram commutes.
\begin{equation}\label{fig:diagram}
\begin{tikzcd}
\hbox{SDE for X} \arrow{r}{\hbox{It\^o's lemma}} \arrow[swap]{d}{\hbox{Draw}} & \hbox{SDE for $f(X)$} \arrow[swap]{d}{\hbox{Draw}} \\
\hbox{Picture of SDE for X in $\R^n$} \arrow{r}{f} & f( \hbox{Picture of SDE for X})
\end{tikzcd}
\end{equation}
Moreover, we will show how the language of $2$-jets allows us to write a particularly elegant formulation of It\^o's lemma.

In Section \ref{differentialGeometrySection} we describe the relationship between the jet formulation and differential operator formulations
of SDEs.
We use the language of jets to give geometric expressions
for many important concepts that arise in stochastic analysis. These geometric representations are in many ways more elegant than the traditional representations in terms of the coefficients of SDEs. In particular we will give coordinate free formulations of the following: the diffusion operators; It\^o SDEs on manifolds and Brownian motion
on Riemannian manifolds.

In Section \ref{drawingVectorSDEs}, we return to the question of graphical representations of SDEs. We show how to represent processes driven by multiple Brownian motions. We illustrate this using the Heston stochastic volatility model (two-dimensional diffusion) and Brownian motion on the torus.

In Section \ref{sec:strato} we consider how our formulation is related to the
Stratonovich formulation of SDEs. We will prove that sections of the bundle of $n$-jets of curves
in a manifold correspond naturally to $n$-tuples of vector fields in the manifold. When translated into a statement about SDEs, the special
case when $n=2$ can be interpreted as the correspondence between It\^o calculus and Stratonovich calculus. 

In Section \ref{fanDiagramSection} we consider an alternative approach to
understanding the coefficients of 1-dimensional SDEs based on the coordinate free notion
of percentiles. We will see that the $2$-jet defining an It\^o SDE can be interpreted as
defining a {\em fan diagram} showing the limiting trajectories of certain percentiles of
the probability distributions associated with the SDE solution process. Moreover we will show that the drift of the Stratonovich formulation can be similarly interpreted as a short-time approximation
to the median. We also consider short time behaviour of the mode.

Our work has a number of applications. Firstly graphical representations of SDEs
should be a valuable tool for the qualitative analysis of SDEs and for developing
an intuitive understanding of the properties of SDEs. Our illustrations of
It\^o's lemma give a first example of this. Secondly coordinate free formulations of 
SDEs will often be considerably simpler than 
local coordinate formulations and hence should assist in the theoretical development of stochastic differential geometry. An example of this is given in \cite{armstrongBrigoOptimalProjectionArxiv,armstrongbrigoicms} where a
notion of projection for SDEs is defined using the $2$-jet approach. It is considerably easier to
understand this notion using jets than with a local coordinate formulation.
A further application is given in \cite{armstrongMijatovic} where the jet approach is used to numerically solve SDEs on manifolds. We hope in future work to give applications of this method to statistics similar to those given in \cite{byrneGirolami}.

\ifarxiv{}{\pagebreak}

\section{SDEs as fields of curves driven by a single Brownian motion}
\label{mainSection}

\subsection{Drawing and simulating SDEs as ``fields of curves"}

Suppose that at every point $x$ in $\R^n$ we have an associated smooth curve
\[
\gamma_x: \R \to \R^n \quad \hbox{ with } \gamma_x(0)=x.
\]
As an example we might define $\gamma^E_x$ on $\R^2$ as follows
\[ \gamma^E_{(x_1,x_2)}(t) = (x_1,x_2) + t(-x_2,x_1) + 3 t^2(x_1,x_2 ). \]
We will use the superscript $E$ to indicate this example curve throughout. This field of curves is plotted in Figure \ref{fig:exampleSDE}.

To be precise we
have taken a grid of points in $\R^2$ which are marked as dots in the figure.
We have then drawn the curve $\gamma^E_x$
at each grid point $x$ for the parameter values $t$ in $(-0.1, 0.1)$.
In general when
drawing such a figure for a general $\gamma$, one should use the same
range $t \in (-\epsilon,\epsilon)$ for every curve in the figure, but one is
free to choose $\epsilon$ to make the diagram visually appealing. (In just the same
way when drawing vector fields, one chooses a sensible scale for each vector).

\begin{figure}[ht]
\centering
\begin{minipage}[m]{0.3\linewidth}
\includegraphics[width=0.95\linewidth]{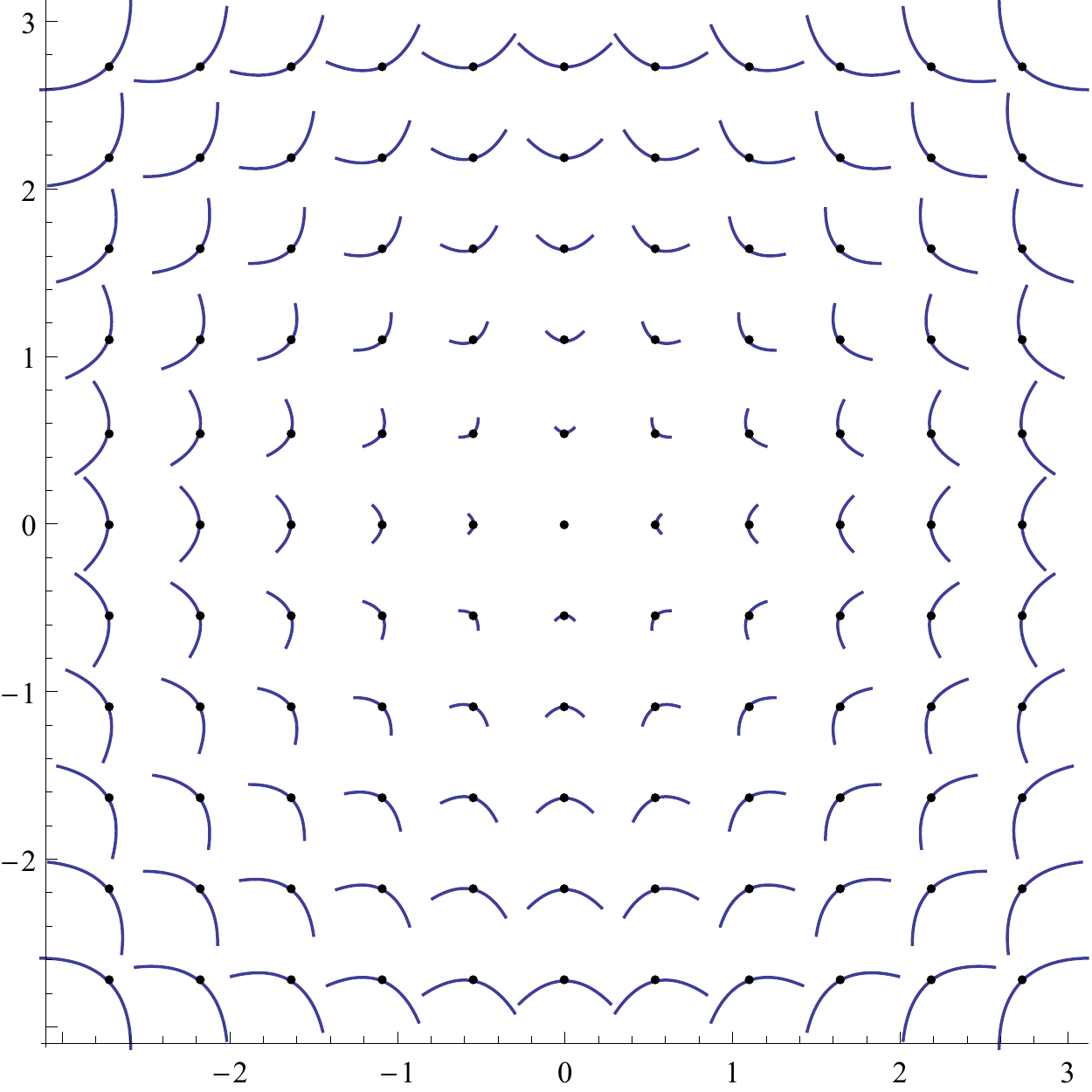}
\caption{A plot of $\gamma^E$}
\label{fig:exampleSDE}
\end{minipage}
\hfill
\begin{minipage}[m]{0.65\linewidth}
\begin{tabular}{cc}
\includegraphics[width=0.45\linewidth]{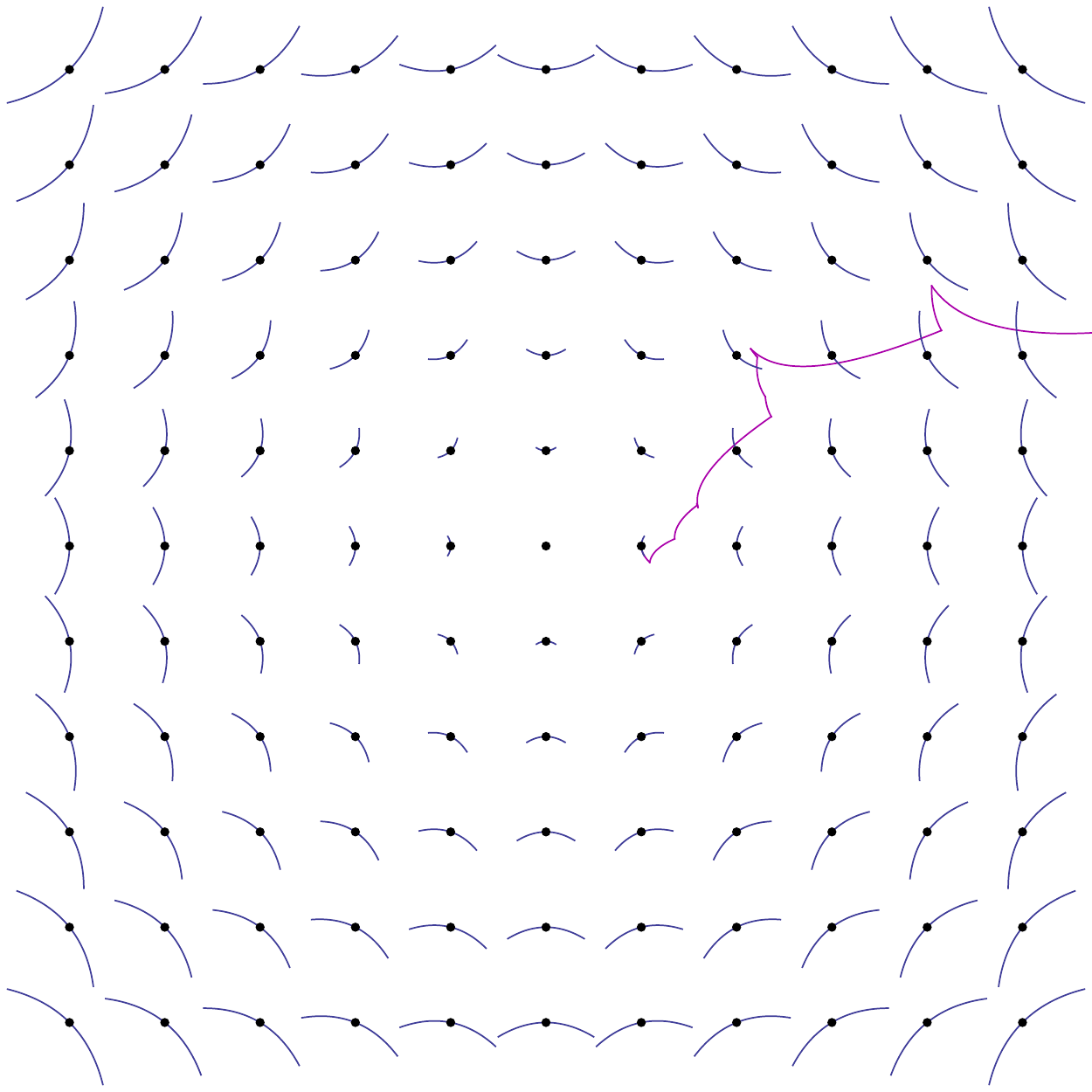} &  \includegraphics[width=0.45\linewidth]{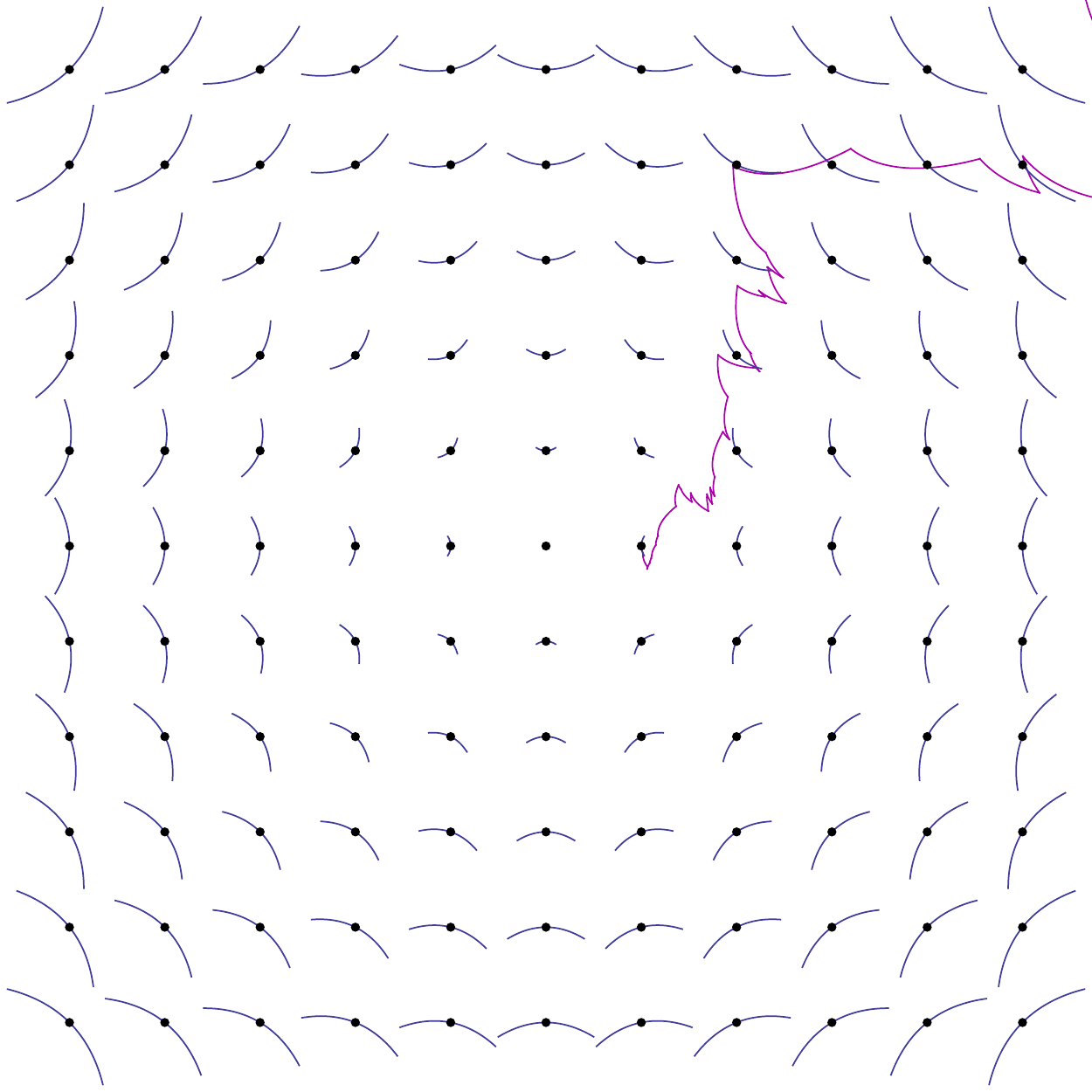} \\
$\delta t = 0.2 \times 2^{-5}$ & $\delta t = 0.2 \times 2^{-7}$ \\
\includegraphics[width=0.45\linewidth]{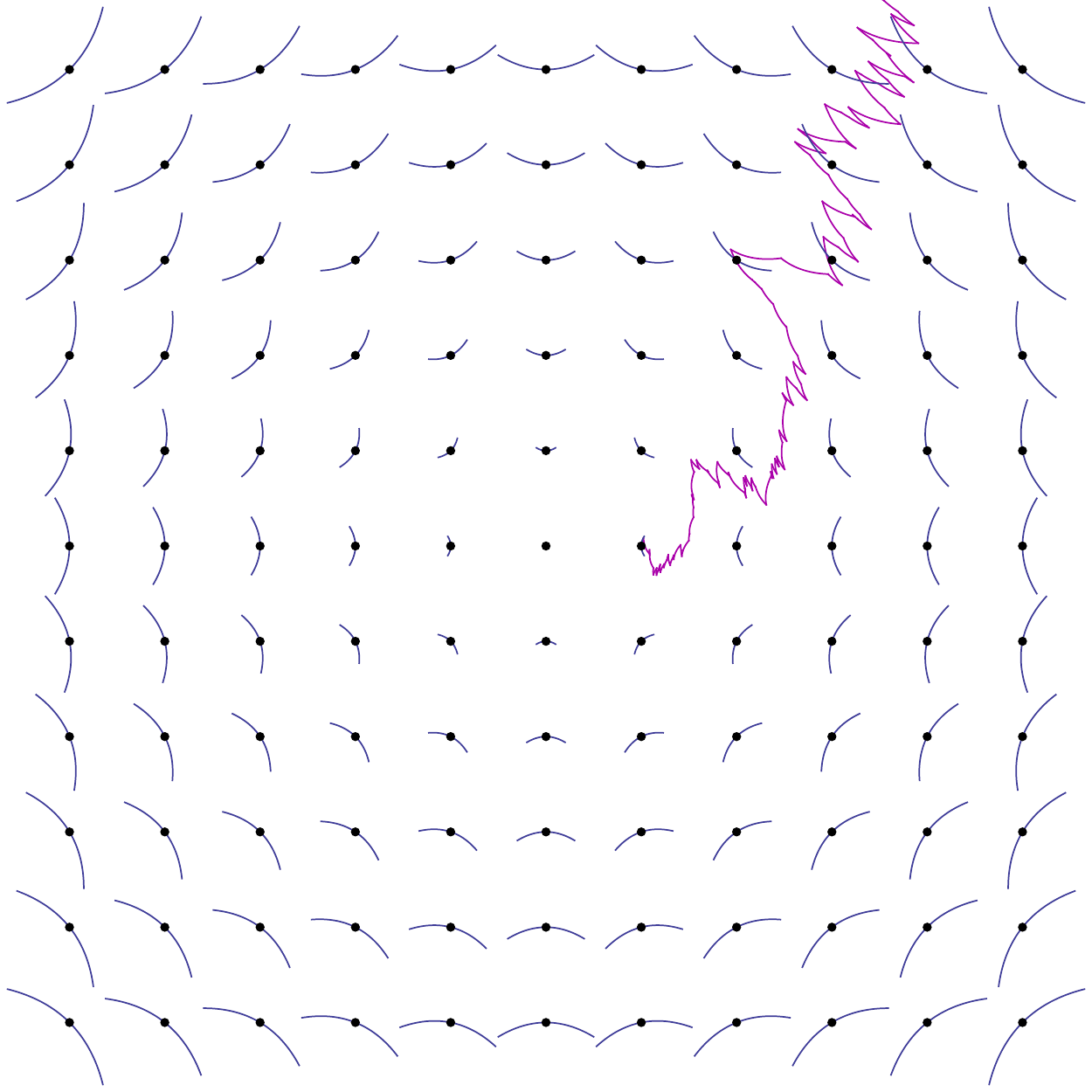} &  \includegraphics[width=0.45\linewidth]{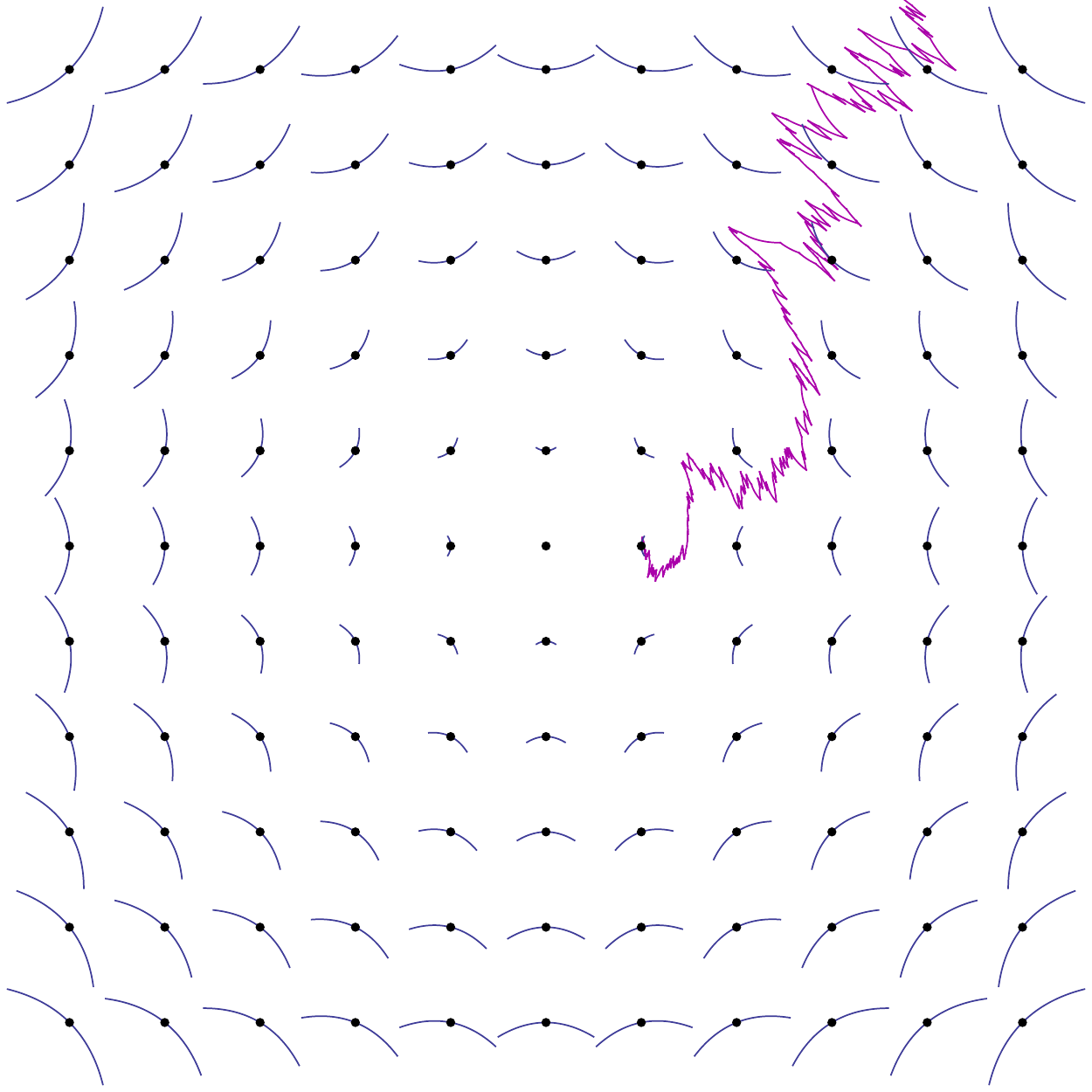} \\
$\delta t = 0.2 \times 2^{-9}$ & $\delta t = 0.2 \times 2^{-11}$ \\
\end{tabular}
\caption{Discrete time trajectories for $\gamma^E$ for a fixed $W_t$ and $X_0$ with different
values for $\delta t$}
\label{illustration}
\end{minipage}
\end{figure}

As can be seen
in the figure, our specific example, $\gamma^E$ has a circular symmetry. This arises from the radially outward $t^2$ component and the orthogonal counterclockwise circular component $t$. Our example has also been chosen to have zero derivatives with respect to $t$ from the third derivative on. This is because we will show how to define a stochastic process in terms of a field of curves $\gamma$ and we will see that the limiting behaviour of this process only depends upon the first and second order terms in $t$.

Given such a $\gamma$, a starting point $X_0$ (the deterministic $x_0 = (1,0)$ 
 in our example, $X_0=x_0$), a Brownian motion $W_t$ and a time step $\delta t$ we
can define a discrete time stochastic process using the following recurrence relation
\begin{equation}\label{discreteTimeEquation}
X_0 := x_0, \ \ \ 
X_{t+\delta t} :=  \gamma_{X_t}( W_{t+\delta t} - W_t ).
\end{equation}

In Figure \ref{illustration} we have plotted the trajectories of process for $\gamma^E$, the starting point $(1,0)$,
a fixed realization of Brownian motion and a number of different time steps. Rather than just
plotting a discrete set of points for this discrete time process, we have connected the points
using the curves in $\gamma^E_{X_t}$.

Notice that since the $\delta W_t = W_{t+\delta t} - W_t$ are normally distributed with standard deviation $\sqrt{\delta t}$
we can interpret the trajectories as being randomly generated trajectories that move
from $X_t$ to $X_{t+\delta t}$ by following the curve $s \mapsto \gamma_{X_t}(s)$ from $s=0$ to $s=\epsilon_t \sqrt{\delta t} $ where the $\epsilon_t$ are independent normally distributed random variables.

As the figure suggests,
these discrete time stochastic processes \eqref{discreteTimeEquation} converge in some sense to a limit as the time step tends to zero. 

We will use the following notation for the limiting process.
\begin{equation}
\mbox{Coordinate free SDE:} \ \ \ X_t \jet \gamma_{X_t}( \ed W_t ), \quad X_0 = x_0. 
\label{continuousTimeEquation}
\end{equation}
For the time being, let us simply treat equation \eqref{continuousTimeEquation} as a short-hand way
of saying that equation \eqref{discreteTimeEquation} converges in some sense to a limit. Note
that it will not converge for arbitrary $\gamma$'s but it does converge for nice $\gamma$ such as $\gamma^E$ or $\gamma$'s with sufficiently good regularity.
The reader familiar with It\^o calculus will want to know how this notation corresponds to It\^o stochastic differential equations and in precisely what sense  and under
what circumstances equation \eqref{discreteTimeEquation} converges to a limit. These questions
are addressed in Section \ref{sec:sdes2jets}.

An important feature of equation \eqref{discreteTimeEquation} is that it makes no reference
to the vector space structure of $\R^n$ for our state space $X$. We have maintained this in the formal notation used
in equation \eqref{continuousTimeEquation}. By avoiding using the vector space structure on
$\R^n$ we will be able to obtain a coordinate free understanding of stochastic differential equations.

\begin{example}
\label{powerExamples}
For a fixed $\alpha \in \mathbb{N}$, in a given coordinate system on $\R$, we can define curves at each point in $\R$ by:
\[
\gamma^\alpha_x(s) = x+s^\alpha
\]
Let us compute the limit of the discrete time process corresponding to these curves.
In the case $\alpha=1$, we have trivially that the $X_t=x_0+W_t$. 
By equation \eqref{discreteTimeEquation} we have that
\[
\begin{split}
X_{n\delta t} &= x_0 + \sum_{i=1}^n (W_{(i+1)\delta t}-W_{i \delta t})^\alpha \\
&= x_0 + (\delta t)^\frac{\alpha}{2} \sum_{i=1}^n \epsilon_i^\alpha
\end{split}
\]
where the $\epsilon_i$ are independently normally distributed with mean $0$ and standard deviation $1$.
Fixing a terminal time $T$ so that $\delta t=\frac{T}{n}$ we have
\[
X_{T} =
x_0 + (T/n)^\frac{\alpha}{2} \sum_{i=1}^n \epsilon_i^\alpha.
\]
By the strong law of large numbers we see that if $\alpha=2$ this converges a.s.  to $x_0+T$. If $\alpha\geq3$ we see that this converges a.s. to  $x_0$.
\end{example}

\subsection{SDEs as fields of curves up to order 2: 2-jets}\label{sec:sdes2jets}

Let us now invoke explicitly the $\R^n$ structure of the state space by choosing a specific coordinate system and consider the (component-wise) Taylor expansion of $\gamma_x$. We have
\[
\gamma_x(t) =  x + \gamma^\prime_x(0) t + \frac{1}{2}\gamma^{\prime \prime}_x(0) t^2 + R_x t^3, \ \ R_x = \frac{1}{6} \gamma_x^{\prime \prime \prime}(\xi), \ \ \xi \in [0,t],
\]
where $R_x t^3$ is the remainder term in Lagrange form. 
Substituting this Taylor expansion in our Equation (\ref{discreteTimeEquation}) we obtain
\begin{equation}
\delta X_t = \gamma^\prime_{X_t}(0) \delta W_t + \frac{1}{2}\gamma^{\prime \prime}_{X_t}(0) (\delta W_t)^2 +    R_{X_t} (\delta W_t)^3, \quad X_0 = x_0.
\label{continuousTimeEquationW2withR}
\end{equation}

Example \ref{powerExamples} suggests that we can replace the term $(\delta W_t)^2$ with $\delta t$
and we can ignore terms of order $(\delta W_t)^3$ and above. So we expect that under reasonable conditions, in the chosen coordinate system, the recurrence relation
given by \eqref{discreteTimeEquation}
will converge to the same limit as the numerical scheme
\[
\delta \bx_t =  \gamma^\prime_{\bx_t}(0) \delta W_t + \frac{1}{2}\gamma^{\prime \prime}_{\bx_t}(0) \delta t, \quad \bx_0 = x_0.
\]
Defining $a(X):= \gamma_X^{\prime \prime}(0)/2$ and $b(X):=\gamma_X^\prime(0)$ we have that this last equation can be written as
\begin{equation}
\delta \bx_t =  a(\bx_t) \delta t  + b(\bx_t) \delta W_t . 
\label{eulerScheme}
\end{equation}
It is well known that this last scheme (Euler scheme) does converge in some appropriate sense to a limit (\cite{kloedenAndPlaten}). This limit is more conventionally written as the solution to the It\^o stochastic differential equation
\begin{equation}
\ed \tx_t =  a(\tx_t) \, \ed t  + b(\tx_t) \ed W_t, \quad \tx_0 = x_0.
\label{itoSDE}
\end{equation}
The coefficient $a(\tx_t)$ is often referred to as the {\em drift}. The coefficient $b(\tx_t)$ is often referred to as the {\em diffusion coefficient} (also known as volatility in applications to social sciences).
Thus, given a coordinate system, we may think of equation \eqref{discreteTimeEquation} as defining a numerical scheme for approximating the It\^o SDE \eqref{itoSDE}. 
In this context we call \eqref{discreteTimeEquation} the $2$-jet scheme.  A rigorous proof of the convergence of the $2$-jet scheme
in mean square ($L^2(\mathbb{P})$) to the solution of the It\^o SDE, based on appropriate bounds on the derivatives of the curves $\gamma_x$,
is given in Appendix \ref{proofsSection}. This notion of convergence is  not fully coordinate independent, however in Appendix \ref{coordinateFreeConvergence} we describe a fully coordinate free notion of convergence which we call mean square convergence on compacts. Our proof of convergence in $L^2(\mathbb{P})$ imples that the $2$-jet scheme will always converge in mean square on compacts if the coefficients are sufficiently smooth.

At this point one may wonder in which sense Equation \eqref{discreteTimeEquation} and its limit are coordinate free. 
It is important to note that the coefficients of equation \eqref{itoSDE} only depend upon the first two derivatives of $\gamma$.
We say that two smooth curves $\gamma:\R \to \R^n$ have the same $k$-jet ($k \in \mathbb{N}, k>0$) if
they are equal up to order $O(t^k)$  in a given coordinate system. If this holds in a given coordinate system, it will hold in all coordinate systems. More generally we have:
\begin{definition} A $k$-jet of a function between smooth manifolds $M$ and $N$ is defined to be the equivalence class of all smooth maps $f:M\to N$ that are equal up to order $k$ in one, and hence all, coordinate systems.
\end{definition}
%Jets represent abstract Taylor expansions that allow one to approach differential equations on manifolds in a coordinate-free manner. Also, jets and the related notions of obstruction and symbol are powerful tools in the analysis of differential equations in a geometric setting \cite{saunders}. 

%This is similar to what is done to define tangent vectors, leading to a coordinate free definition.
Using this terminology, we say that the coefficients of equation \eqref{itoSDE} (and \eqref{eulerScheme}) are determined by the $2$-jet of a curve $\gamma:\R \to \R^n$ in a specific coordinate system. 
%Jets are defined in much more general terms on manifolds in differential geometry.
%, for example they can be defined as equivalence classes of maps with the same Taylor expansion up to a given order in given (and hence all) charts. 

In the light of the above convergence result, we can say that in pictures such as Figure \ref{fig:exampleSDE} one should avoid interpreting any details other than the first two derivatives of the curve. One way of doing this is by insisting that we draw the quadratic curves that best fit the curves $\gamma$ rather than the actual curve $\gamma$ itself. 

Notice that vectors can be defined in the same way as $1$-jets of smooth curves. In just the
same way as we draw quadratic curves in Figure \ref{fig:exampleSDE}, one normally chooses to draw straight lines in a diagram of a vector field.

In summary: vector fields are
fields of $1$-jets and they represent ODE's; diagrams such as Figure \ref{fig:exampleSDE} are pictures of fields of $2$-jets and they represent It\^o SDEs.

Given a curve $\gamma_x$, we will write $j_2(\gamma_x)$ for the two jet associated with $\gamma_x$. This is formally defined to be the equivalence class of all curves which are equal to $\gamma_x$ up to $O(t^2)$ included.

Since we will show that, under reasonable regularity conditions, the limit of the symbolic equation \eqref{continuousTimeEquation} depends only on the $2$-jet of the driving curve, we may rewrite equation \eqref{continuousTimeEquation} as
\begin{equation}
\mbox{Coordinate-free 2-jet SDE:} \hspace{1cm} X_t \jet j_2(\gamma_{X_t})( \ed W_t ), \quad X_0 = x_0. 
\label{continuousTimeEquationJet}
\end{equation}
This may be interpreted either as a coordinate free notation for the classical It\^o SDE given by equation \eqref{itoSDE}
or as a shorthand notation for the limit of the process given by the discrete time equation \eqref{discreteTimeEquation}.

\subsection{Coordinate-free It\^o formula with jets}

Suppose that $f$ is a smooth mapping from $\R^n$ to itself and suppose that $X$ satisfies
(\ref{discreteTimeEquation}). It follows that $f(X)$ satisfies
\[
 (f(X))_{t+\delta t} = f \circ \gamma_{X_t}( \delta W_t ) .
\]
Taking the limit as $\delta t$ tends to zero we have:
\begin{lemma}\label{lemma:ito}[It\^o's lemma --- coordinate free formulation]
If the process $X_t$ satisfies
\[
X_t \jet j_2(\gamma_{X_t})( \ed W_t )
\]
then, writing $\times$ for the Cartesian product of functions,  $(X_t,f(X_t))$ satisfies
\[
(X_t,f(X)_t) \jet j_2((\gamma_{X_t} \times f \circ \gamma_{X_t}))( \ed W_t )  .
\]
We might also write more directly, with abuse of notation, $f(X)_t \jet j_2( f \circ \gamma_{X_t})( \ed W_t )$.
\end{lemma}
If one prefers the more traditional format for SDEs given in \eqref{itoSDE} we simply
need to calculate the derivatives of $f \circ \gamma$. In a chosen coordinate system, let us write $\gamma_X^i$ for
the $i$-th component of $\gamma_X$ with respect to the coordinates $x_1, x_2, \ldots x_n$
for $\R^n$. Two applications of the chain rule give
\[
\begin{split}
(f \circ \gamma_X)^\prime(t) &= \sum_{i=1}^n \frac{\partial f}{\partial x_i}( \gamma_X(t)) \frac{\ed \gamma_X }{\ed t}  \\
(f \circ \gamma_X)^{\prime \prime}(t) &= \sum_{j=1}^n \sum_{i=1}^n \frac{\partial^2 f}{\partial x_i \partial x_j}( \gamma_X(t)) \frac{\ed \gamma^i_X }{\ed t} \frac{\ed \gamma^j_X }{\ed t} \\
&\quad + \sum_{i=1}^n \frac{\partial f}{\partial x_i}( \gamma_X(t)) \frac{\ed^2 \gamma_X }{\ed t^2}
\end{split}
\]
We conclude that our lemma is equivalent to the classical It\^o's lemma.

{\emph{We can now interpret It\^o's lemma geometrically as the statement that the transformation
rule for jets under a change of coordinates is given by composition of functions.}}

Since we now understand the geometric content of It\^o's lemma, we can draw a picture to illustrate it. Consider the transformation $\phi:\R^2 / \{0\} \to [-\pi,\pi] \times \R$ by
\[(\theta,s) = \phi(x_1,x_2)  = \left(\mbox{arctan}(x_2/x_1), \log(\sqrt{x_1^2+x_2^2})\right),\]
or equivalently \[
\phi( \exp(s) \cos(\theta), \exp(s) \sin(\theta) ) = (\theta,s),
\]
applied to our example process $\gamma^E$. This can be viewed as a transformation of the complex plane $\phi(z)=i \log(z)$. 
We use $\phi$ to transform the bottom right picture in Figure \ref{illustration}
in two different ways. Firstly we apply directly $\phi$ to each point of Figure \ref{illustration} to obtain a new point to be inserted in a new figure. This is done by using image manipulation software. In other words we stretch the image without any consideration of its mathematical structure.  The result of this is shown in the left hand side of Figure \ref{transformedPictures}.

\begin{figure}[ht]
\centering
\includegraphics[width=0.4\linewidth]{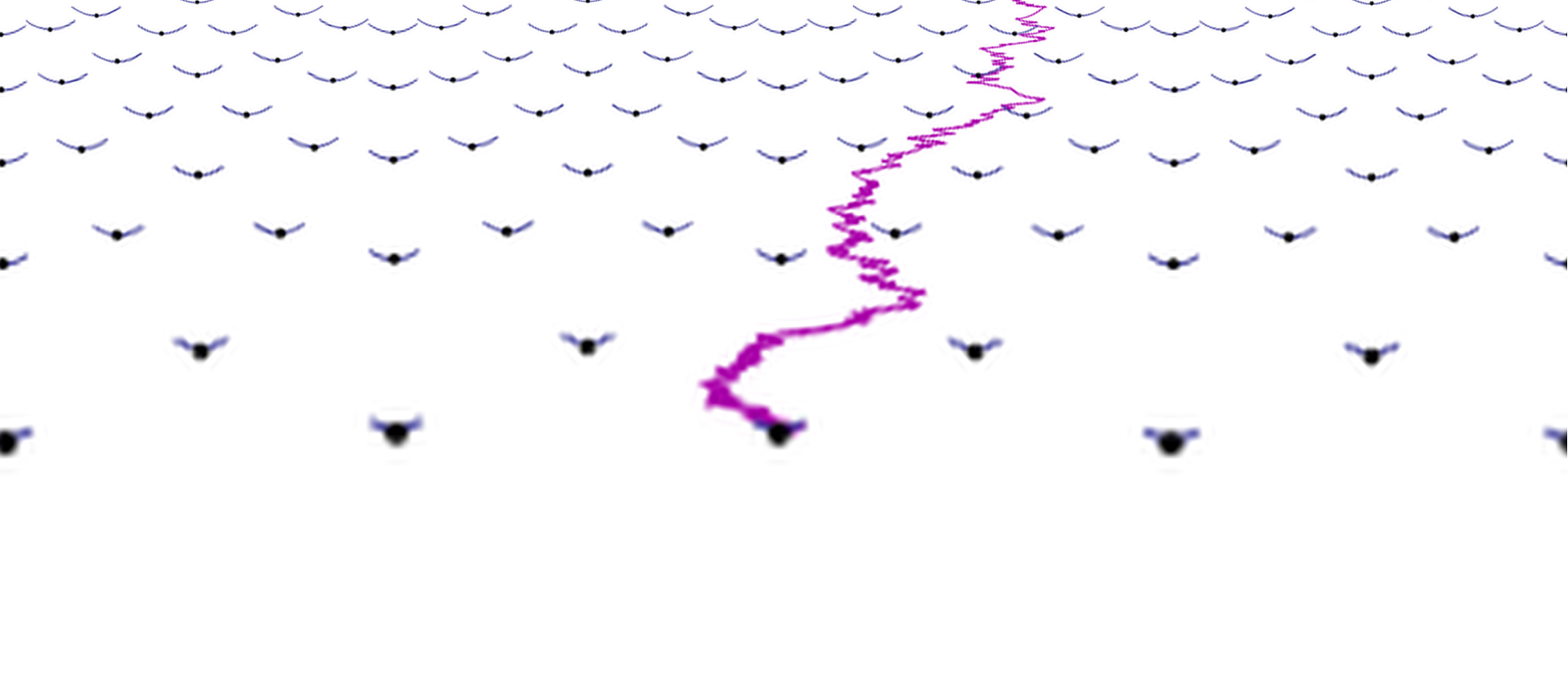}
\includegraphics[width=0.4\linewidth]{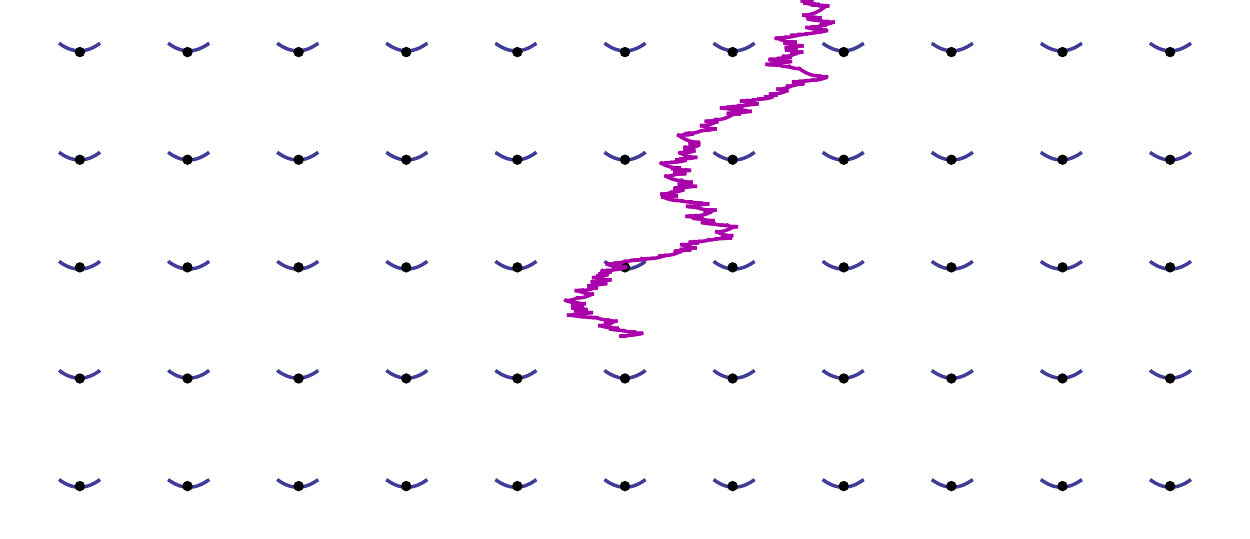} 
\caption{Two plots of the process $j_2(\phi \circ \gamma^E)$ in the plane $(\theta,s)$. The left plot was generated by transforming the image
in Cartesian coordinates pixel by pixel. The right was computed using It\^o's lemma.}
\label{transformedPictures}
\end{figure}

As an alternative approach, we transform our equation using It\^o's lemma applied
to the function $\phi$. So equation \eqref{eq:exampleSDECartesianCoordinates} below for $(X_1, X_2)$ transforms to the equation \eqref{eq:exampleSDEPolarCoordinates} for $(\theta, s)$.
\begin{equation}
d \horzvect{X_1}{X_2} = 3 \horzvect{X_1}{X_2} dt + \horzvect{-X_2}{X_1} dW_t,
\label{eq:exampleSDECartesianCoordinates}
\end{equation}
\begin{equation}
d \horzvect{\theta}{s} = 3 \horzvect{0}{\frac{7}{2}} dt + \horzvect{1}{0} dW_t.
\label{eq:exampleSDEPolarCoordinates}
\end{equation}
We can then use this equation to plot the process $(\theta, s)$ directly by simulating the process in discrete
time as before. The result is shown in the right hand side of Figure \ref{transformedPictures}.

As one can see the two approaches to plotting the transformed process
give essentially identical results, showing an example of our earlier diagram \eqref{fig:diagram} at work. The differences one can see are:
the lower quality in the left image, obtained by transforming pixels rather than using vector graphics; the grid points
at which the $2$-jets are plotted are changed; small differences in the simulated
path since we have only simulated discrete time paths.

We have assumed that the $2$-jet $j_2(\gamma_x)$
is associated in a deterministic and time independent manner with the point $x$.
However, we are investigating how the theory can be generalized to time dependent and stochastic choices of $2$-jets.

\subsection{SDEs driven by vector-Brownian motion as 2-jets}\label{sec:sdesmultiplebrownians}
Consider jets of functions of the form
\[
\gamma_x: \R^d \to \R^n.
\]
Just as before we can consider discrete time difference equations of the form
\begin{equation}
 X_{t+\delta t} := \gamma_{X_t}\left( \delta W^1_t, \ldots, \delta W^d_t 
\right),
\label{vectorJetEquation}
\end{equation}
or, if writing $\delta W_t^\alpha = \epsilon_\alpha \sqrt{\delta t}$, with $\epsilon$ independent normals,  
\[  X_{t+\delta t} := \gamma_{X_t}\left( \epsilon_1 \sqrt{\delta t}, \ldots, \epsilon_d \sqrt{\delta t} \right) .  \]

Again, the limiting behaviour of such difference equations will only depend upon the $2$-jet $j_2(\gamma_x)$ and can be denoted by \eqref{continuousTimeEquationJet}, where it is now understood that $d W_t$ is the  vector Brownian motion increment.

The multi-dimensional analogue of Example \ref{powerExamples} suggests we write $\delta W^\alpha \delta W^\beta \approx g_E^{\alpha \beta} \delta t$. Here $g_E^{\alpha \beta}$ denotes the Euclidean metric on $\R^d$. Thus $g_E^{\alpha \beta}$ is equal to $1$ if $\alpha$ equals $\beta$ and $0$ otherwise. We can now compute a second
order Taylor expansion as follows.

\begin{equation}
\delta X_t \approx \frac{1}{2} \sum_\alpha \sum_\beta \frac{ \partial^2 \gamma_{X_t}}{\partial u^\alpha \partial u^\beta} 
g^{\alpha \beta}_E \delta t + \sum_{\alpha} \frac{\partial \gamma_{X_t}}{\partial u^\alpha}
\delta W^\alpha_t.
\label{notEinstein}
\end{equation}
Here $u^\alpha$ are the standard orthonormal coordinates for $\R^d$. 
We have chosen to write $g_E^{\alpha \beta}$ instead of using a Kronecker $\delta$ because one might want to choose non-orthonormal coordinates for $\R^d$ and so it is useful to notice that $g_E$ transforms as a symmetric $2$-form and not an endomorphism. Another advantage is that we can use the Einstein summation convention. For example \eqref{notEinstein} can be rewritten as
\begin{equation}
\begin{split}
\delta X^i_t 
&= \frac{1}{2}  \partial_\alpha \partial_\beta \gamma^i
g_E^{\alpha \beta} \delta t  + \partial_\alpha \gamma^i  \, \delta W^\alpha_t.
\label{sdeEinstein}
\end{split}
\end{equation}
We have shown informally how to define an SDE as the limit of a numerical scheme defined in terms of $2$-jets and we have shown how this scheme can
be written in local coordinates. A reader who is familiar with classical It\^o
calculus will immediately recognize \eqref{sdeEinstein} as the Euler scheme
for the It\^o stochastic differential equation obtained by replacing each
$\delta$ in \eqref{sdeEinstein} with a $\ed$. We now state the relationship
between these two approaches formally.

\begin{theorem}\label{th:convergence} {\bf Convergence of the 2-jet schemes to It\^{o} SDEs.} 
Let $\gamma_x: \R^d \to \R^n$ be a smoothly varying family of functions
whose first and second derivatives in $\R^d$ satisfy Lipschitz conditions (and hence linear growth
bounds). In other words we require that there exists a positive constant $K$ such that for all
$x, y \in \R^n$ and $\alpha, \beta \in \{1, 2, \ldots, d\}$ we have
\[
\begin{split}
 \left| \frac{\partial \gamma_x}{\partial u^\alpha}\Bigr|_{u=0} - \frac{\partial \gamma_y}{\partial u^\alpha}\Bigr|_{u=0} \right| &\leq K |x-y|, \\
 \left| \frac{\partial^2 \gamma_x}{\partial u^\alpha \partial u^\beta}\Bigr|_{u=0} - \frac{\partial \gamma_y}{\partial u^\alpha \partial u^\beta}\Bigr|_{u=0} \right| &\leq K |x-y|, \\
\bigg{(} \mbox{and hence} \ \  \left| \frac{\partial \gamma_x}{\partial u^\alpha}\Bigr|_{u=0}  \right|^2 &\leq K^2 \left(1 + |x|^2\right),  \\
 \left| \frac{\partial^2 \gamma_x}{\partial u^\alpha \partial u^\beta}\Bigr|_{u=0}  \right|^2 &\leq K^2 \left( 1 + |x|^2 \right) \ \ \ \bigg{)}. \\
\end{split}
\]
Note that we are using the letter $u$ to denote the standard orthonormal coordinates for $\R^d$.
Suppose in addition that we have a uniform bound on the third derivatives:
\[
\left| \frac{ \partial^3 \gamma_x }{\partial u^\alpha \partial u^\beta \partial u^\gamma }\Bigr|_{u=0}
\right| \leq K .
\]
Let $T$ be a fixed time and let ${\cal T}^N := \{0,\delta t, 2\,\delta t, \ldots, N\,\delta t =T\}$, 
be a set of discrete time points.
Let $X^N_t$ (we will omit the superscript below) denote the $2$-jet scheme defined by 
\begin{equation}\label{eq:schemediscretemultiW}
X_{t+\epsilon} = \gamma_{X_t}\left( \frac{\epsilon}{\delta t}(W_{t+\delta t}-W_t) \right), \qquad  t \in {\cal T}^{N-1},\ \epsilon \in [0,\delta t]\ , \qquad X_0=x_0.
\end{equation}
This converges in $L^2(\Pb)$ to the classical It\^o solution of the corresponding SDE, namely
\begin{equation}
\tx_t = \tx_0 + \int_0^t a(\tx_s) \, \ed s + \sum_{\alpha=1}^d \int_0^t b_\alpha(\tx_s) \, \ed W^{\alpha}_s,
\quad t \in [0,T]
\label{eq:equivalentClassical}
\end{equation}
where
\[
a(x) := \frac{1}{2} \sum_{\alpha=1}^d \frac{\partial^2 \gamma_x}{\partial u^\alpha \partial u^\alpha} \Bigr|_{u=0}
\]
and
\[
b_\alpha(x) := \frac{\partial \gamma_x}{\partial u^\alpha} \Bigr|_{u=0}.
\]
More precisely we have the estimate
\begin{equation}\label{sup-msconv}
\sup_{t \in [0,T]} E\left\{ |X^{N}_t - \tx_t|^2 \right\} \le C\ \delta t = C \frac{T}{N}
\end{equation}
for some constant $C$ independent of $N$. We denote the coordinate free equation obtained as limit of \eqref{eq:schemediscretemultiW} by
\[ X_t \jet j_2(\gamma_{X_t})(dW_t) . \]
\end{theorem}

This theorem proves one of the main results of this paper: a It\^o SDE can be represented in a coordinate free manner simply as a 2-jet driven by Brownian motion.  The proof is given in Appendix \ref{proofsSection}.

\section{Jets and second order operators}
\label{differentialGeometrySection}

\begin{definition}{\bf (Coordinate free It\^o SDEs driven by vector Brownian motion).}
An It\^o SDE or It\^o diffusion on a manifold $M$ is a section of the bundle of $2$-jets of maps $\R^d \rightarrow M$ together with $d$ Brownian motions $W^i_t$, $i=1,\ldots,d$.
\end{definition}

Our discrete time formulation (equation \eqref{vectorJetEquation}) 
of an It\^o SDE on a manifold is already coordinate free in that it makes no mention of the vector space structure of $\R^n$. Given such a $\gamma$ we can
write down a corresponding It\^o SDE using the notation of equation \eqref{continuousTimeEquationJet}.
This may be interpreted either as indicating the limit of the discrete time process or as a short-hand
for a classical It\^o SDE. The reformulation of It\^o's lemma in the language of jets shows that this second interpretation will be independent of the choice of coordinates. The only issue one needs to consider are the bounds needed to ensure existence of solutions. The details of transferring the theory of existence and uniqueness of solutions of SDEs to manifolds are considered in, for example, \cite{elworthy}, \cite{elworthybook}, \cite{ito2}, \cite{emery} and \cite{hsu}.
The point we wish to emphasize is that the coordinate free formulation of an SDE given in equation
\eqref{continuousTimeEquationJet} can be interpreted just as easily on a manifold as on $\R^n$.

We will now see how this allows us to study the differential geometry of SDEs
in a coordinate free manner. In particular we will show how to give coordinate free definitions of key concepts such as the diffusion operators and Brownian motion.

Note that an alternative approach to stochastic differential geometry is to place these operators centre stage. This is the essence of the approach of second order tangent vectors and Schwarz morphisms. Thus this section can also
be seen as establishing the relationship between these approaches.
\ifarxiv{See Appendix \ref{appendix:equivalentFormulations} for further details.}{See Appendix B of the preprint version \cite{arxivVersion} for further details.}

Suppose that $f$ is a function mapping $M$ to $\R$. We can define a differential operator acting
on $functions$ in terms of a $2$-jet associated with $\gamma_x$ as follows.
\begin{definition}{\bf (Backward diffusion operator via 2-jets).} The Backward diffusion operator for the It\^o SDE corresponding to $\gamma_x$ is defined on suitable  functions $f$ as
\begin{equation}
{\cal L}_{\gamma_x} f := \frac{1}{2}\Delta_E (f \circ \gamma_x){\Bigr|_0} = \frac{1}{2} \partial_\alpha \partial_\beta (f \circ \gamma_x){\Bigr|_0}  g_E^{\alpha \beta}.
\label{backwardDiffusion}
\end{equation}
Here $\Delta_E$ is the Laplacian defined on $\R^d$.
${\cal L}_{\gamma_x}$ acts on functions defined on $M$.
\end{definition}
In contexts where the SDE is understood we will simply write ${\cal L}$.

Note that this definition is simply the drift term of the It\^o SDE for $f(X)$ computed using It\^o's lemma. To first order, 
the drift measures how the expectation of an SDE solution process
changes over time. Thus, with $\delta$ denoting the forward $t$ increment as usual,
\begin{equation}
\left(\delta \E[f(X_t)|X_s=x]\right)|_{s=t}= ({\cal L}_{\gamma_x}(f))\delta t + O(\delta t^2).
\label{fokkerPlanckInfinitessimal}
\end{equation}

Before proceeding to define the forward diffusion operator, let us briefly recall
the theory of the bundle of densities on a manifold.

Recall that given a vector space $V$ and a group homomorphism
\[ \tau:GL(n,\R) \to \hbox{Aut}(V) \]
we can define an associated bundle $\bf{V}$ over $M$. The fibres of the bundle
over a point $p$ are given by equivalence classes of charts $\phi:M\to\R^n$ and vectors
$v \in V$ under the relation
\[
(\phi_1,v_1) \sim ( \phi_2, v_2) \quad \Leftrightarrow \quad \tau\left((\phi_1)_* \circ (\phi_2)^{-1}_*\right)_p(v_2)=v_1.
\]
Each chart $\phi:M\to\R^n$ defines local coordinates on $\bf{V}$. We use this to define
the smooth structure on $\bf{V}$. This generalizes straightforwardly \cite{kobayashiAndNomizu}
to allow one to associate a vector bundle
with any principal $G$-bundle and a representation of the Lie group $G$. In our special case, the
principal $G$-bundle we are using is the frame bundle of the manifold.

Now consider the representation
\[ \tau(g) = |\det(g)| \in \hbox{Aut}(\R). \]
This defines a bundle over a manifold $M$ called the bundle of densities. This bundle is denoted $\Vol$.
The usual transformation formula for probability densities under changes of coordinates tells us
that a probability density over $M$ is a section of $\Vol$.

Integration defines a pairing between functions and densities on $M$ by
\[
\begin{split}
\int:& \ \ \ \Gamma(\R) \times \Gamma(\Vol) \to \R \\
\hbox{by }& \int( f, \rho ) := \int f \rho.
\end{split}
\]
On non-compact manifolds one must either insist that one of $f$ or $\rho$ is compactly
supported or consider decay rates of $f$ and $\rho$ to ensure this is well defined.

Note that from a probabilistic point of view, this pairing is interpreted as taking expectations.

Using integration by parts, we can define ${\cal L^*}$ to be the formal adjoint of ${\cal L}$
with respect to this pairing. This is called the {\em forward diffusion operator}.

Even if one does not know the initial state $X_0$ but knows its probability density, $\rho$, one
may integrate equation \eqref{fokkerPlanckInfinitessimal} to obtain
\[
\frac{ \partial}{\partial t} \int_M f(\rho) = \int_M ({\cal L}f)(\rho).
\]
Since this holds for all smooth compactly supported functions $f$ we can deduce
\[
\frac{ \partial \rho}{\partial t} = {\cal L^*}(\rho).
\]
We conclude that the Fokker--Planck equation follows from It\^o's lemma for functions.

Notice that both ${\cal L}$ and ${\cal L}^*$ are linear second order operators. The key difference
is that they have different domains. ${\cal L}$ acts on functions; ${\cal L}^*$ acts on densities.
This gives a geometric explanation as to why ${\cal L}$ appears in the Feynman--Kac equation
which tells us about the evolution of expectations of functions whereas ${\cal L}^*$ appears in the Fokker--Planck equation which tells us about the evolution of probability densities.

\subsection{Weak and strong equivalence}

We see that both the It\^o SDE \eqref{sdeEinstein} and the backward diffusion operator 
\eqref{backwardDiffusion} use only
part of the information contained in the $2$-jet. Specifically only the diagonal terms of $\partial_\alpha \partial_\beta \gamma^i$ (those with $\alpha=\beta$) influence the SDE and even for these terms it is only their average value that is important. The same consideration applies to
the backward diffusion operator. This motivates the following

\begin{definition}
We say that two $2$-jets $\gamma^1_x$ and $\gamma^2_x$
\[ \gamma^i_x: \  \R^d \to M \]
are {\em weakly equivalent} if
\[ {\cal L}_{\gamma^1_x} = {\cal L}_{\gamma^2_x}. \]
We say that $\gamma^1$ and $\gamma^2$ are {\em strongly equivalent} if in addition
\[ j_1(\gamma^1) = j_1(\gamma^2). \]
\end{definition}

We see that the SDEs defined by the two sets of $2$-jets are equivalent if the $2$-jets are strongly equivalent. This means that given the same realization of the driving Brownian motions $W^\alpha_t$ the
solutions of the SDEs will be almost surely the same (under reasonable assumptions to ensure pathwise uniqueness of solutions to the SDEs). 

When the $2$-jets are weakly equivalent, the transition probability distributions resulting from the dynamics of the related SDEs are the same even though the dynamics may be different for any specific realisation of the Brownian motions. For this reason one can define a diffusion process on a manifold
as a smooth selection of a second order linear operator ${\cal L}$ at each point that determines the transition of densities. In this context ${\cal L}$ is known as the {\em infinitesimal generator of the diffusion}. A diffusion can be realised locally as an SDE, but not necessarily globally. 

Recall that the top order term of a quasi linear differential operator is called its symbol. In the case of a second order quasi linear differential operator $D$ which maps $\R$-valued functions
to $\R$-valued functions, the symbol defines a section of $S^2 T$, the bundle of symmetric tensor products of tangent vectors, which we will call $g_D$.

In local coordinates, if the top order term of $D$ is
\[
D f = a^{ij} \partial_i \partial_j f + \hbox{lower order}
\]
then $g_D$ is given by
\[
g_D(X_i, X_j) = a^{ij} X_i X_j.
\]
We are using the letter $g$ to denote the symbol for a second order operator because, in the
event that $g$ is positive definite and $d=\dim M$, $g$ defines a Riemannian metric on $M$. In these circumstances we will say that the SDE/diffusion is {\em non-singular}. Thus we can associate a canonical Riemannian metric $g_{\cal L}$ to any non-singular SDE/diffusion.

\begin{definition}
A  diffusion on a manifold $M$ is called a Riemannian Brownian motion if 
\[ {\cal L}( f )=\frac{1}{2} \Delta_{g_{\cal L}}(f). \]
\end{definition}
Note that given a Riemannian metric $h$ on $M$ there is a unique Riemannian Brownian motion (up to diffusion equivalence) with $g_{\cal L}=h$. This is easily checked with a coordinate calculation.

This completes our definitions of the key concepts in stochastic differential geometry
and indicates some of the important connections between stochastic differential equations,
Riemannian manifolds, second order linear elliptic operators and harmonic maps.

\subsection{Brownian Motion}

Let us examine the important special case of Brownian motion from a variety of perspectives: local coordinates; the exponential map; and mean curvature.

\subsubsection{Local coordinates}

We will now show how the definitions above allow us to compute Brownian motion on a topologically non-trivial $d$-dimensional Riemannian manifold $(M,g)$ for which we have an atlas. We used this computation to generate Figure \ref{fig:pretzelBrownianMotion} where we simulate Brownian motion on a genus $2$ surface.

Let $(x^1, x^2, \ldots, x^d)$ be a chart for $M$ and let $(u^1,u^2, \ldots, u^d)$ be the standard coordinates on $\R^d$ (so $g^{\alpha,\beta}_E$ is the identity matrix). Let $\gamma_x:\R^d \to M$ with $\gamma_x(0)=x$ and let $f:M\to \R$. Let us write $\gamma^i$ for the components of $\gamma$, so $\gamma^i=x^i \circ \gamma$. We compute
\[
\begin{split}
{\cal L}_{\gamma_x} f
&= \frac{1}{2} g^{\alpha \beta}_E \frac{\partial^2}{\partial u^\alpha \partial u^\beta}(f \circ \gamma_x)  \Big|_{u=0} \\
&= \left( \frac{1}{2} g^{\alpha \beta}_E \frac{\partial^2 f}{\partial x^i \partial x^j}(\gamma_x(u)) \frac{\partial \gamma^i_x}{\partial u^\alpha} \frac{\partial \gamma^j_x}{\partial u^\beta}  \right) \Bigg|_{u=0}
+ \left( \frac{1}{2} g^{\alpha \beta}_E \frac{\partial f}{\partial x^i}(\gamma_x(u)) \frac{\partial^2 \gamma^i_x}{\partial u^\alpha \partial u^\beta}  \right) \Bigg|_{u=0}.
\end{split}
\]
Let us assume that $\gamma_x$ is a quadratic function of $u$ so that in local coordinates
\[
\gamma_x^i = x + b^i_\alpha u^\alpha + a^{i}_{\alpha \beta} u^\alpha u^\beta
\]
where $b^i$ and $a^{i}_{\alpha \beta}$ are real valued functions on the manifold and $i,\alpha,\beta \in \{1, \ldots d \}$. So
\[
{\cal L}_{\gamma_x} f
= \frac{1}{2} g^{\alpha \beta}_E b^i_\alpha b^j_\beta \frac{\partial^2 f}{\partial x^i \partial x^j} +  g^{\alpha \beta}_E a^i_{\alpha \beta} \frac{\partial f}{\partial x^i}.
\]
Following standard conventions, on a Riemannian manifold we write $g_{ij}$ for the metric tensor in local coordinates, we write $|g|$ as an abbreviation for $\det g_{ij}$ and we write $g^{ij}$ for the inverse matrix of $g_{ij}$. The Laplacian on $(M,g)$ is then given in local coordinates by
\[
\begin{split}
\Delta f &= \frac{1}{\sqrt{|g|}} \frac{\partial}{\partial x^i}
\left( \sqrt{|g|} g^{ij} \frac{\partial f}{\partial x^j} \right) \\
&= \frac{1}{\sqrt{|g|}} \frac{\partial}{\partial x^i}
\left( \sqrt{|g|} g^{ij} \right) \frac{\partial f}{\partial x^j}
+ g^{ij} \frac{\partial^2 f}{\partial x^i \partial x^j},
\end{split}
\]
(see, for example, \cite{jost}).
We have said that $\gamma$ defines Brownian motion on $(M,g)$ if ${\cal L}_{\gamma_x}$ is equal to half the Laplacian operator. So $\gamma$ will define Brownian motion if at each point $x$,
\begin{equation}
g^{\alpha \beta}_E b^i_\alpha b^j_\beta = g^{ij}
\label{eqn:pseudoSquareRoot}
\end{equation}
\begin{equation}
g^{\alpha \beta}_E a^i_{\alpha \beta} = \frac{1}{2 \sqrt{|g|}} \frac{\partial}{\partial x^j}
\left( \sqrt{|g|} g^{ij} \right).
\label{eqn:brownianDrift}
\end{equation}
As we would expect, these equations are under-determined. We can find a solution
to \eqref{eqn:pseudoSquareRoot} by taking the matrix $b^{i}_\alpha$ to 
be the Cholesky decomposition of the matrix $g^{ij}$. We can also find a solution of
equation \eqref{eqn:brownianDrift} by taking
\[
a^{i}_{\alpha \beta} = \begin{cases}
\frac{1}{2 d \sqrt{|g|}} \frac{\partial }{\partial x^j}( \sqrt{|g|} g^{ij})
& \text{if } \alpha=\beta \\
0 & \text{otherwise}.	
\end{cases}
\]
In summary we have found a canonical choice of $\gamma$ that locally defines Brownian motion in a chart. Given an atlas we can then choose $\gamma_x$ at each point by choosing a $\gamma_x$ from one of the charts around $x$. Although $\gamma_x$ itself will not vary smoothly between charts, the weak equivalence class of $\gamma_x$ will vary smoothly.

In Figure \ref{fig:pretzelBrownianMotion} we show the result of simulating Brownian motion on the genus $2$ surface in $\R^3$ given in coordinates $(y_1, y_2, y_3)$ by
\[
((y_1 - 1) y_1^2 (y_1 + 1) + y_2^2 )^2 +  y_3^2 = \frac{1}{30}.
\]
We found $14$ charts for this surface by projecting along each of the axes $y_i$ (Mathematica's Solve function made this easy to do). At each point $x$ in this manifold, we chose a specific one of these charts containing $x$ by projecting along the axis whose inner product with the normal at $x$ had the largest absolute value (we preferred the axis $y_i$ with the lowest index $i$ in the event of a tie). In this way we were able to define an explicit quadratic map $\gamma_x$ at each point. Since the image of $\gamma_x(u)$ will leave the chart for large $u$, we defined a new function by $\tilde{\gamma}(x) = \gamma_x(u) \rho(| \gamma_x(u)| )$ where $\rho:[0,\infty]\to[0,\epsilon)$ is a smooth increasing function equal to the identity near $0$ and where the value $\epsilon$ was so as to ensure that $\tilde{\gamma}$ would never leave the selected chart around $x$. By construction $j_2(\tilde{\gamma})=j_2(\gamma)$, so the maps $\tilde{\gamma}$ can be used to approximate Brownian motion in discrete time using equation \eqref{vectorJetEquation}.

\subsubsection{Exponential map}

Another choice of canonical map $\gamma_x:\R^d \to M$ that generates Brownian motion on a Riemannian manifold is the exponential map \cite{jost}.
In this case the convergence of the discrete time scheme \eqref{vectorJetEquation} is well-known. In those situations where the exponential map can be calculated explicitly this gives the most obvious choice of $\gamma$. Our local coordinate calculation shows how Brownian motion can be simulated when the exponential map is not explicitly known.

We note that simulating Brownian motion on a Riemannian manifold provides a useful tool for sampling from probability distributions on that manifold. For example, \cite{byrneGirolami} uses the exponential map to simulate Brownian motion and provides explicit coordinate calculations in the case of Stiefel manifolds and show with examples how this can be applied to statistical problems such as dimension-reduction. The papers \cite{girolamiCalderhead, diaconisEtAl, brubakerEtAl} discuss related approaches to sampling from manifolds.

One does not need a metric to define the exponential map, it can be defined using a connection alone. The scheme \eqref{vectorJetEquation} resulting from the exponential map of a connection was studied in \cite{gangolli} and a version of Theorem \ref{th:convergence} was proved for this case (see also the similar paper \cite{pinsky}). This approach to defining stochastic differential equations on manifolds is known as the McKean--Gangolli injection.

\subsubsection{Mean Curvature}

Brownian motion on a hypersurface $H$ in $\R^{d+1}$ can be defined using a $d$-dimensional stochastic process in $\R^{d+1}$ that has the property that trajectories which start on $H$ stay on $H$. Let us see how this can be understood geometrically with jets.	
	
Suppose that we have a map $\gamma_x:\R^d \to \R^{d+1}$ at each point $x$ of $H$. Choose orthonormal coordinates $u^\alpha$ for $\R^d$. For each $\alpha=1 \ldots d$, consider the curves $\gamma^\alpha_x:\R \to \R^{d+1}$ by $\gamma^\alpha_x = \gamma \circ i^\alpha$ where $i^\alpha:\R \to \R^d$ is the inclusion given by coordinate $\alpha$. If $\gamma_x$ were the exponential map, each $\gamma_x^\alpha$ would be parametrised by arc-length and would have curvature orthogonal to $H$. Moreover the mean of these curvatures as $\alpha$ varies from $1$ to $d$ would be the mean curvature vector of $H$. In coordinates, if the $\gamma^\alpha_x$ are parameterized by arc-length the mean of the curvatures is given by
\[
\frac{1}{d} g^{\alpha \beta}_E \partial_{\alpha} \partial_{\beta} \gamma_x^i.
\]	
We note that this quantity is not dependent on the choice of coordinates $u^\alpha$ and is equal to $\frac{2}{d}$ times the drift term in \eqref{sdeEinstein}.

From our discussion of weak equivalence of $2$-jets, we deduce that $\gamma_x$ defines Brownian motion on a hyper-surface $H$
if and only if both: to first order $\gamma_x$ is an isometry onto the tangent space of $H$; the mean of the curvatures of $\gamma$ is
equal to the mean-curvature vector of $H$. This gives a geometric interpretation of the drift of Brownian motion. \ifarxiv{Note that the word mean in mean-curvature has nothing to do with the stochastic process. It simply refers to the average of the curvatures across all principle directions.}{}

Note that the second condition is somewhat stronger than requiring that the image of $\gamma$ has the same mean-curvature vector as $H$. This is because the mean-curvature vector is always orthogonal to $H$ but the mean of the curvatures depends upon the parameterization and may have a component tangent to $H$.	    

In general an SDE on a manifold $M \subseteq \R^n$ can be understood as SDEs on $\R^n$ whose trajectories which start on $M$ remain on $M$. As in the above example, at each point in the manifold $M$ the $2$-jet in $\R^n$ will be completely determined by the SDE on $M$. One can interpret the drift term of \eqref{sdeEinstein} for the process in $\R^n$ as having a tangent component determined by the intrinsic SDE on $M$ and an orthonormal component determined by the requirement that the trajectories are confined to $M$.

\section{Drawing SDEs driven by vector Brownian motion}
\label{drawingVectorSDEs}

We can draw an It\^o SDE driven by $d$-dimensional Brownian motion
by drawing a function
\[
\gamma_x: \R^d \to M
\]
at every point on the manifold $M$. Of course, in practice one only draws the function at
a finite set of sample points in $M$.

However,
$\gamma_x$ is not uniquely determined by the SDE. By
drawing $\gamma_x$, we are drawing a representative of the equivalence class of two jets that define the same SDE. To illustrate this, in the top line of 
Figure \ref{fig:equivalentTwoJets} we have plotted three functions $\gamma^*_0:\R^2 \to M$ whose $2$-jets all define the same SDE. They are defined as follows.
\begin{equation}
\begin{split}
\gamma^A_0(x,y)&= x(1,0) + 2 y(0,1) + 2 x^2(1,0), \\
\gamma^B_0(x,y)&= x(1,0) + 2 y(0,1) + 2 y^2(1,0), \\
\gamma^C_0(x,y)&= x(1,0) + 2 y(0,1) + (x^2 + y^2)(1,0).
\end{split}
\label{equivalentJets}
\end{equation}
We have plotted the image of a circle of radius $\epsilon=0.3$ in $\R^2$ under each of the maps
$\gamma^*_0$. The grid lines shown are the image of polar grid lines rather than Cartesian grid lines.
Polar coordinates are a more natural choice for plotting SDEs since the rotational symmetry of $\R^d$
corresponds to the notion of weak equivalence of SDEs.

\begin{figure}[htb]
\begin{center}
\begin{tabular}{ccc}
\includegraphics[width=0.20\linewidth]{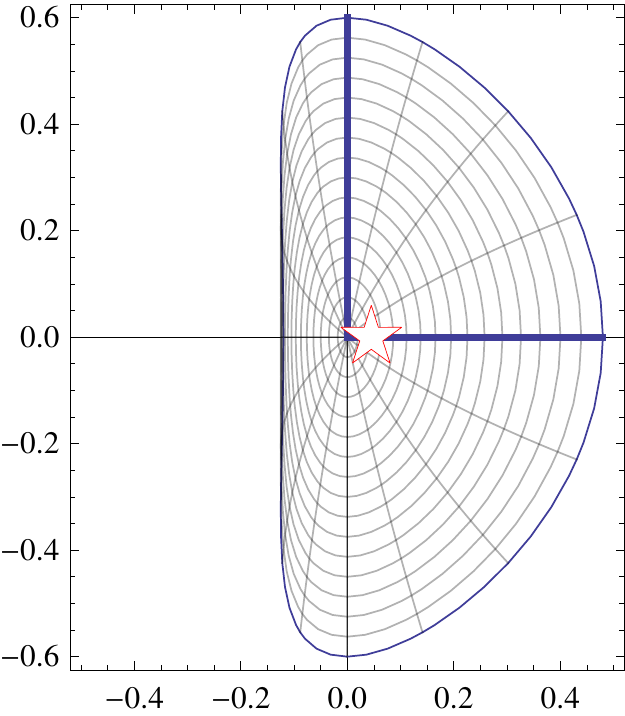} &
\includegraphics[width=0.20\linewidth]{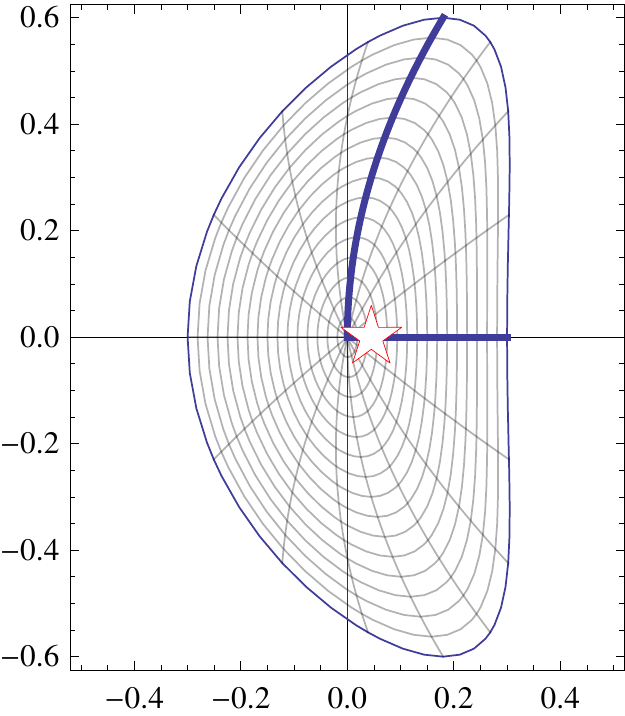} &
\includegraphics[width=0.20\linewidth]{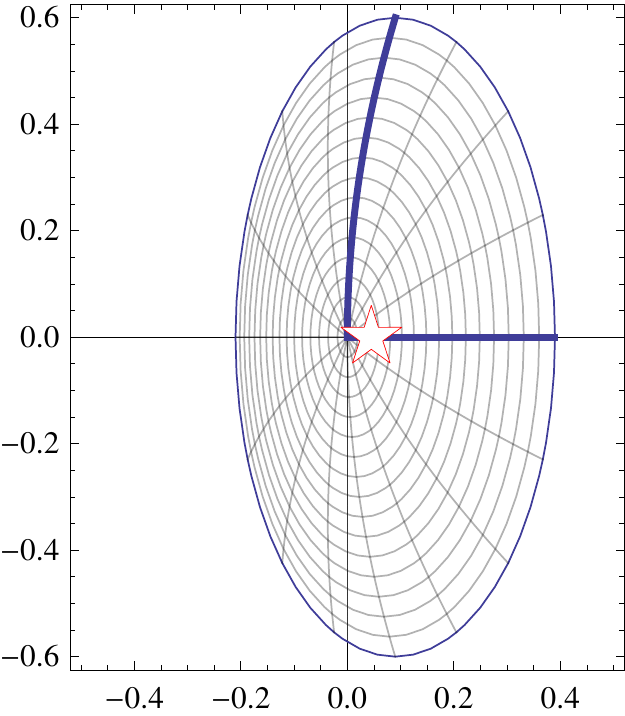} \\
\includegraphics[width=0.20\linewidth]{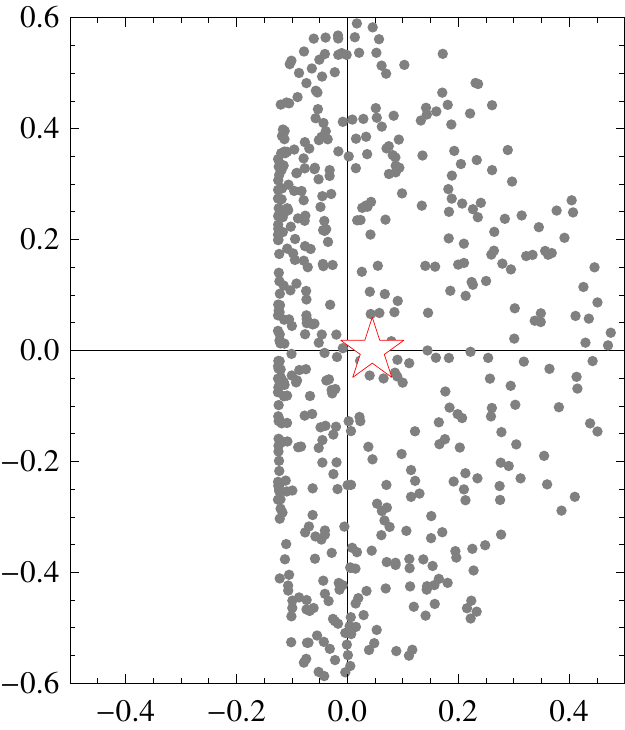} &
\includegraphics[width=0.20\linewidth]{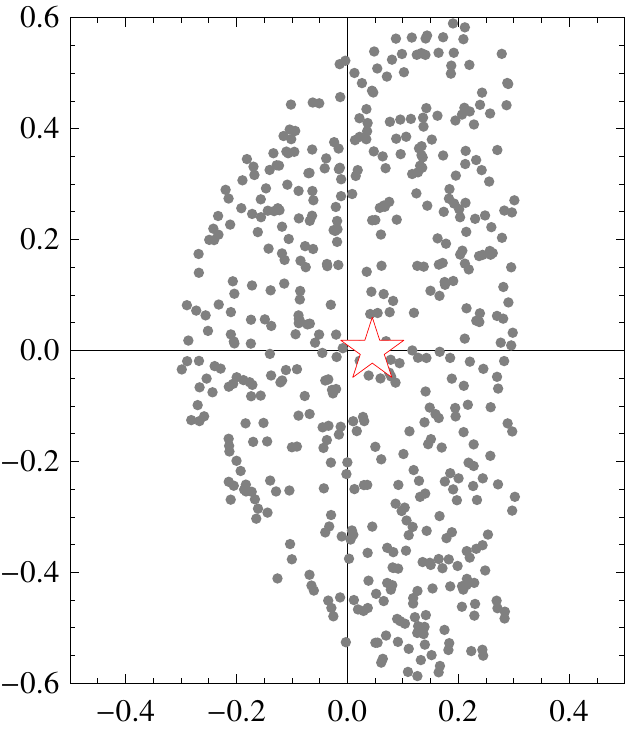} &
\includegraphics[width=0.20\linewidth]{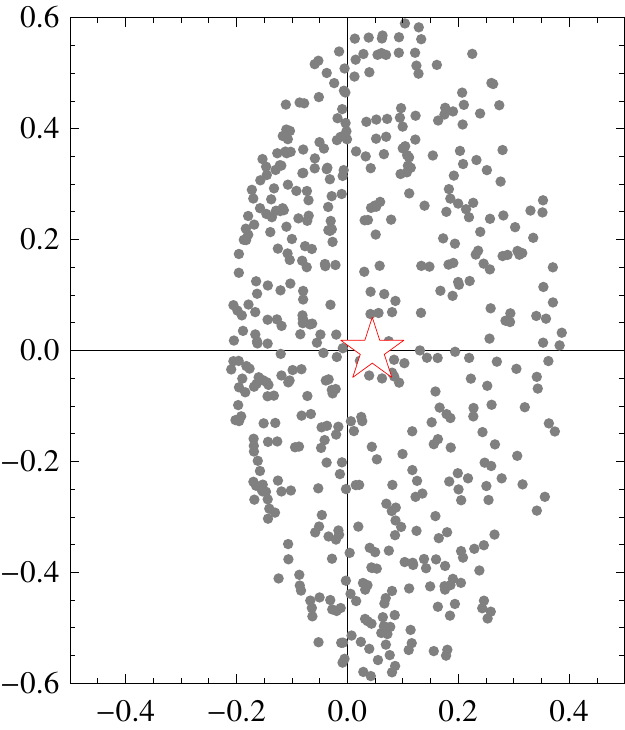} \\
$\gamma^A_0$ & $\gamma^B_0$ & $\gamma^C_0$
\end{tabular}
\end{center}
\caption{Plots of the equivalent $2$-jets $\gamma^*_0$ defined in equation \ref{equivalentJets}}
\label{fig:equivalentTwoJets}
\end{figure}

The diffusion term of an SDE corresponds to the first order term of the jets. 
This is a linear mapping of the plane and hence maps the unit circle to an ellipse. This gives rise
to the broadly elliptical shape of the plots.
The drift term of the SDE corresponds to the mean of the image. This
is marked with a star. As one can see this drift is the same for all the plots $\gamma^*_0$.

If we are interested in strong equivalence of the SDEs,
the image of each axis is important as it tells us the strength of each component of the Brownian
motion. It is only the direction of the axes that is important and not the curvature. We have used thicker
lines to indicate the image of the $x$ and $y$ axes under each map in \ref{fig:equivalentTwoJets}.

An alternative plot of the $2$-jets $\gamma^*_0$ is shown in the bottom line of Figure \ref{fig:equivalentTwoJets}. Instead of plotting the jet by showing the image of polar grid lines,
we have shown the image of a set of $1000$ uniformly distributed points inside a ball of
radius $\epsilon=0.3$. We have again plotted the mean point with a star. These plots
eliminate the extraneous details of the maps $\gamma_X$ allowing one to see more clearly
the key features that determine the weak equivalence class of the jets. These plots provide a clear visual
link between the geometric and the probabilistic properties of the SDE.

If one wished to illustrate strong equivalence, the plots on the lower line in Figure \ref{fig:equivalentTwoJets}
could be augmented with vectors indicating the mappings of each axis up to first order. Such a diagram
would illustrate the key features of the strong equivalence class stripped of visual distractions such
as the curvature of the images of the axes.

One can, therefore, draw a two dimensional SDE by drawing an infinitesimal diagram
of the sort shown in Figure \ref{fig:equivalentTwoJets} at a number of points. These drawings
would satisfy the desirable commutativity property illustrated in diagram \eqref{fig:diagram}. However, the resulting drawings
would be very busy. 

We can strip out some of the excessive detail from our diagram by deciding to choose a specific representative of the $2$-jet at each point. Given an SDE in local coordinates,
\[ \ed X_t = a(X_t) \ed t + b_i(X_t) \ed W^i_t \]
we choose the specific two jet given by
\[ \gamma_x(s) = x + \frac{1}{d} a \, g^E_{ij} s^i s^j + b_i s^i. \]
The image of an $\epsilon$ ball under $\gamma_x$ will be an ellipsoid. Moreover, if we
know that $\gamma_x$ is of this form, we can recover the coefficients $a$ and $b_i$ up
to weak equivalence just from knowledge of the image of the $\epsilon$ ball. As an example,
notice that the curve $\gamma^C_0$ from \eqref{equivalentJets} is of this form.
This allows us to simplify our diagrams by representing each $2$-jet by
drawing the image of an $\epsilon$ ball at each point under this specific
representative of the $2$-jet.

For example in Figure \ref{fig:heston} we show a plot of the Heston stochastic volatility model
with drift (see \cite{heston}):
\begin{equation}
\begin{split}
\ed S_t &= \mu S_t \ed t + \sqrt{\nu_t} S_t \ed W^1_t, \\
\ed \nu_t &= \kappa( \theta - \nu_t) \ed t+ \xi \sqrt{\nu_t} \, (\rho \ed W^1_t + \sqrt{1-\rho^2} \ed W^2_t ).
\end{split}
\label{hestonEquation}
\end{equation}
Note that as well as plotting the ellipses, the figure indicates the exact point that each ellipse is associated with. The extent to which the centre of the ellipse differs from the associated point is a measure of the drift.

Figure \ref{fig:heston} is coordinate dependent since its definition depends upon choosing a specific representative of the $2$-jet at each point. However,  it can be thought of as a visual shorthand for a coordinate independent diagram where repeated copies of the more detailed pictures of Figure \ref{fig:equivalentTwoJets} are used.

\begin{figure}[ht]
\centering
\begin{minipage}[t]{0.45\textwidth}
\includegraphics[width=1\linewidth]{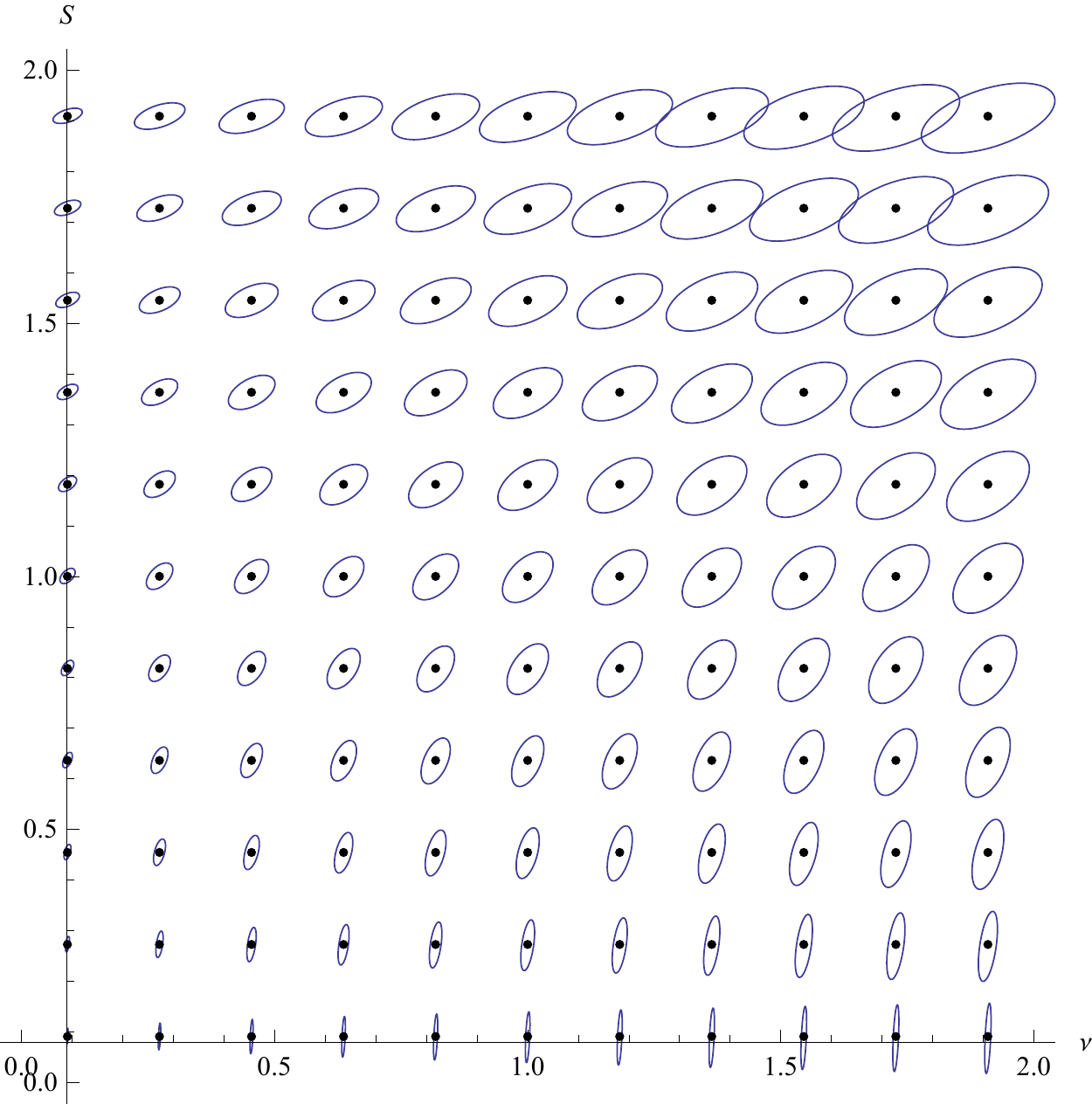}
\caption{A plot of the Heston model \eqref{hestonEquation}. Parameter values 
$\xi = 1$,
$\theta = 0.4$,
$\kappa = 1$,
$\mu = 0.1$,
$\rho = 0.5$. We have
plotted the image of the $\epsilon$-balls for $\epsilon=0.05$}
\label{fig:heston}
\end{minipage} %
\hfill %
\begin{minipage}[t]{0.45\textwidth}
\includegraphics[width=1\linewidth]{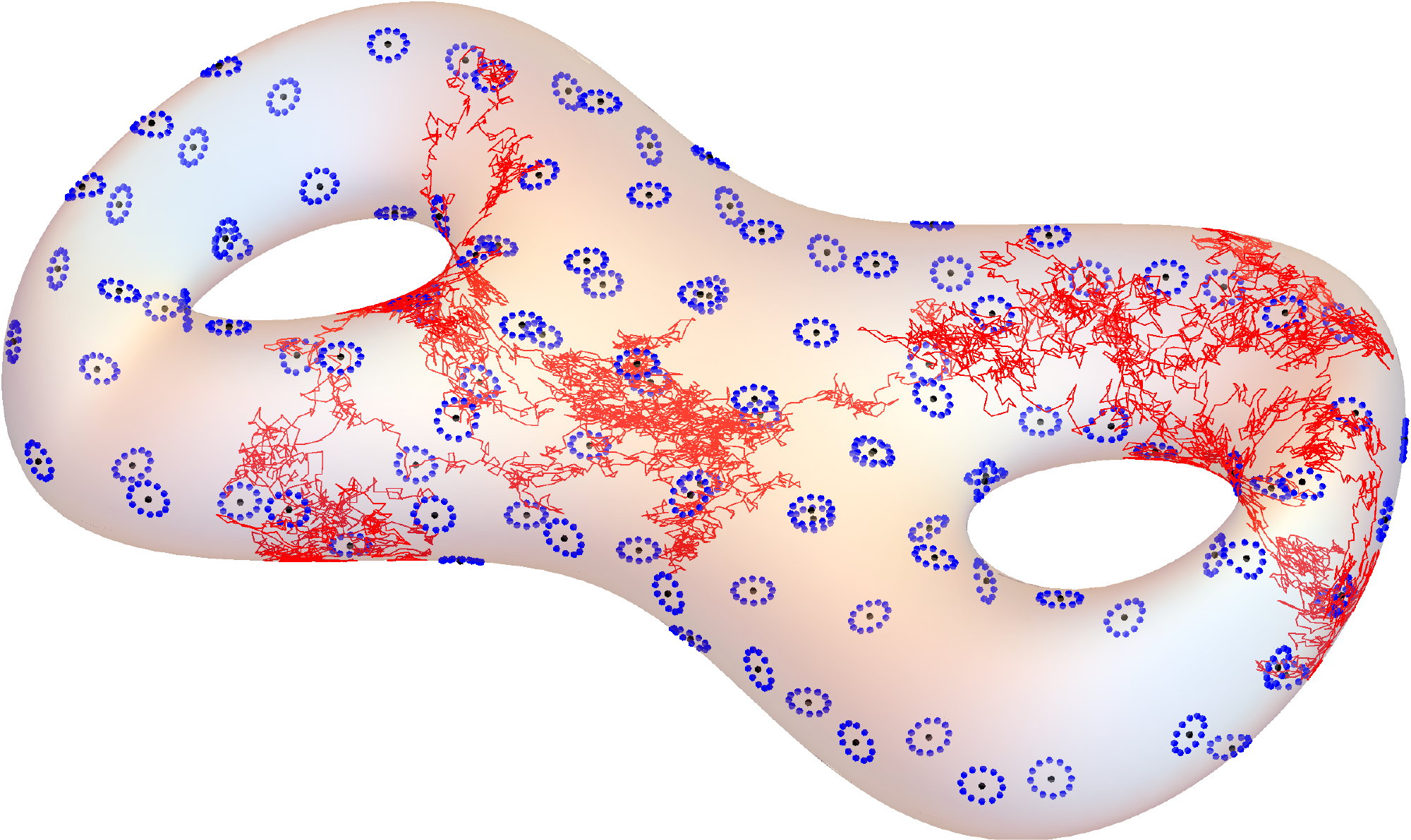}
\caption{Plot of Brownian motion over a short, finite time interval on a genus $2$ surface (red). The blue dots can be interpreted as a plot of the SDE defining Brownian motion.}
\label{fig:pretzelBrownianMotion}
\end{minipage} %
\end{figure}

Similarly Figure \ref{fig:pretzelBrownianMotion} depicts the SDE defining Brownian motion on a genus $2$ surface by showing the image of the center and the $12$ points on the edge of a clock-face under $\gamma$. These are shown in blue. As can be seen all the clock faces appear to be circles of the same size, this is a characteristic property one would expect of Brownian motion. It shows that the metric induced by the SDE is indeed equal to the metric induced by the embedding. Similarly the image of the centre of the clock face appears to be in the centre of each of the circles. This shows that the forward and backward operators are equal.

\section{Jets, vector fields and Stratonovich calculus}\label{sec:strato}
We wish to show how jets can be described using vector fields. This will
allow us to relate our approach to SDEs to the approach of Stratonovich
calculus. 

For simplicity, let us assume in this section that the driver is one dimensional.
Thus to define an SDE on a manifold, one must choose a $2$-jet of a curve at each point of the manifold. One way to specify a $k$-jet of a curve at every point in a neighbourhood is to first choose a chart for the neighbourhood and then consider
curves of the form
\begin{equation}
\gamma_x(t) = x + \sum_{i=1}^k a_i(x) t^i
\label{itoRep}
\end{equation}
where $a_i:\R^n \to \R^n$. As we have already seen 
in Lemma \ref{lemma:ito}, these coefficient functions $a_i$ depend upon the choice of chart in a relatively complex way. For example for $2$-jets the coefficient functions are not vectors but instead transform according to It\^o's lemma. We will call this the {\em standard representation} for a family of $k$-jets.

An alternative way to specify the $k$-jet of a curve at every point is to choose $k$ vector fields $A_1$, \ldots, $A_k$ on the manifold. One can then define $\Phi_{A_i}^t$ to be the vector flow associated with the vector field $A_i$. This allows one to define curves at each point $x$ as follows.
\begin{equation}
\gamma_x(t) = \Phi^{t^k}_{A_k} ( \Phi^{t^{k-1}}_{A_{k-1}} ( \ldots ( \Phi^t_{A_1}(x)) \ldots ))
\label{stratRep}
\end{equation}
where $t^k$ denotes the $k$-th power of $t$.
We will call this the {\em vector representation} for a family of $k$-jets.
It is not immediately clear that all $k$-jets of curves can be written in this way. Let us prove this.

Suppose one chooses a chart and attempts to compute the relationship between the coefficients $a_i$ in the standard representation and the components of the vector fields $A_i$ in the vector representation. 
It is clear that the $O(t)$ term $a_1(x)$ will depend bijectively and linearly on $A_1(x)$. Thus there is a bijection between $1$-jets written in the form \eqref{itoRep} and $1$-jets written in the form \eqref{stratRep}. The $O(t^2)$ term will depend linearly upon $A_2(x)$ together with a more complex term derived from $A_1$ and its first derivative. Symbolically $a_1(x) = \rho(A_2(x)) + f(A_1(x), (\nabla A_1)(x))$ where $\rho$ is a linear bijection determined by the choice of chart. It follows that
there is also a bijection between $2$-jets written in the standard form and $2$-jets written in the vector form. Inductively we have:
\begin{theorem}
Smooth $k$-jets of a curve can be defined uniquely by a list of $k$ vector fields $X_1$, \ldots $X_k$
according to the formula \eqref{stratRep}.
\end{theorem}

Notice that the vector representation specifically allows us to define a {\em family} of $k$-jets varying from point to point. In more technical language, the vector representation allows us to specify a section of the bundle of $k$-jets. If one only specifies vectors at a point rather than vector fields, one cannot define the vector flows and so equation \eqref{stratRep} cannot be used to define a $k$-jet at the point. Thus although there is no natural map from $k$ vectors defined at a point to a $k$-jet of a curve, there is a natural map from $k$ vector fields defined in a neighbourhood to a smoothly varying choice of $k$-jet at each point.

The standard and vector representations simply give us two different coordinate systems for the infinite dimensional space of families of $k$-jets.

When this general theory about $k$-jets is applied to stochastic differential equations one sees
that two corresponding coordinate representations of a stochastic differential equation will emerge. Let us calculate in more detail correspondence between the two representations.

\begin{lemma}
Suppose that a family of $2$-jets of curves is given in the vector representation as
\[
\gamma_x(t) = \Phi_A^{t^2}( \Phi_B^t(x))
\]
for vector fields $A$ and $B$. Choose a coordinate chart and let $A^i$, $B^i$ be the components
of the vector fields in this chart. Then the corresponding standard representation for the family of
$2$-jets is
\[
\gamma_x(t) = x + a(x) t^2  + b(x) t
\]
with
\[
\begin{split}
a^i &= A^i + \frac{1}{2} \frac{\partial B^i}{\partial x^j} B^j\\
b^i &= B^i.
\end{split}
\]
\end{lemma}
\begin{proof}
By definition of the flow $\Phi^t_B$, its components $(\Phi^t_B)^i$ satisfy
\begin{equation}
\frac{ \partial (\Phi^t_B(x))^i }{\partial t} = B^i(\Phi^t_B(x)).
\label{dPhiB}
\end{equation}
Differentiating this we have
\begin{equation}
\begin{split}
\frac{ \partial^2 (\Phi^t_B(x))^i }{\partial t^2} &= \frac{ \partial B^i}{\partial x^j} \frac{\partial (\Phi^t_B)^j}{\partial t} \\
&= B^j \frac{ \partial B^i}{\partial x^j}.
\end{split}
\label{d2PhiB}
\end{equation}
We now compute the derivatives of $\gamma_t(x)$. We write $(\Phi_A^{t^2})^i$ for the $i$-th
component of the function $\Phi_A^{t^2}$.
\[
\frac{\partial}{\partial t} ((\Phi_A^{t^2})^i(\Phi_B^t(x)))
=
2 t \frac{\partial (\Phi^{t^2}_A)^i}{\partial t}( \Phi^t_B )
+ \frac{ \partial (\Phi_A^{t^2})^i }{ \partial x^j}( \Phi^t_B ) \frac{ \partial (\Phi^t_B)^j }{\partial t}
\]
Differentiating this again one obtains
\[
\begin{split}
\frac{\partial^2}{\partial t^2} ((\Phi_A^{t^2})^i(\Phi_B^t(x)))
&=
4 t^2 \frac{\partial^2 (\Phi^{t^2}_A)^i}{\partial t^2}( \Phi^t_B )
+ 2 \frac{\partial (\Phi^{t^2}_A)^i}{\partial t}( \Phi^t_B )
+  4 t \frac{ \partial^2 (\Phi^{t^2}_A)^i}{ \partial x^j \partial t} ( \Phi^t_B ) \frac{ \partial (\Phi^t_B)^j}{\partial t} \\
&\quad
+ \frac{ \partial^2 (\Phi^{t^2}_A)^i}{\partial x^j \partial x^k } ( \Phi^t_B ) \frac{ \partial (\Phi^t_B)^j}{\partial t} \frac{ \partial (\Phi^t_B)^k}{\partial t}
+ \frac{\partial (\Phi^{t^2}_A)^i}{\partial x^j} \frac{\partial \Phi_B^t}{\partial t^2}.
\end{split}
\]
At time $t=0$ we have that $\Phi^0_A$ is simply the identity. So its partial derivatives at time $0$
are trivial to compute. Hence
\[
\begin{split}
\frac{\partial^2}{\partial t^2} ((\Phi_A^{t^2})^i(\Phi_B^t(x)))\Bigr|_{t=0}
&=
2 \frac{\partial (\Phi^{t^2}_A)^i}{\partial t}( \Phi^t_B )
+ \frac{\partial (\Phi_B^t)^i}{\partial t^2} \\
&= 2 A^i + B^j \frac{\partial B^i}{\partial x^j}.
\end{split}
\]
The last line follows from the definition of $\Phi_A^t$ as the flow associated with the vector field $A$ together with equation \eqref{d2PhiB}.
We can now write down the expression for $a(x)$. The expression for $b(x)$ follows immediately from 
equation \eqref{dPhiB}.
\end{proof}

What is interesting about this result is that an SDE
can be defined using the coefficients $a$ and $b$, which transform according
to It\^o's lemma, or they can be defined using vector fields $A$ and $B$, which transform according to the standard chain rule.

An alternative way of showing that SDEs can be defined in terms of vector
fields was already known. It is given by introducing the above-mentioned Stratonovich (or in full Fisk--Stratonovich--McShane 
\cite{fisk} \cite{mcshane} \cite{strato}) calculus. This
provides an alternative to the It\^o calculus of \cite{ito}. The coefficients of SDEs written using Stratonovich calculus transform as vector fields. Indeed these coefficients are precisely the vector fields $A$ and $B$ we have just identified geometrically. Thus we have given a geometric interpretation of how the coordinate free notion of a $2$-jet of a curve is related to the vector fields defining a Stratonovich SDE. This establishes the relationship
between our jet approach to SDEs on manifolds and the more conventional
approach of using Stratonovich calculus. This shows that we may view the choice between It\^o or Stratonovich calculus simply as a choice of coordinates for a single underlying geometric structure.
For readers not familiar with the It\^o and Stratonovich stochastic calculi in a given coordinate system we refer to \ifarxiv{Appendix \ref{appendix:classicalFormulation}.}{Appendix A of the preprint version \cite{arxivVersion}.}

It is worth expanding the discussion on the different stochastic calculi in the light of the above result. 
Despite the initial seminal paper by It\^o \cite{ito2} in a It\^o calculus context, Stratonovich calculus has been the main calculus when interfacing stochastic analysis with differential geometry \cite{elworthy,rogers}. 
We used Stratonovich calculus ourselves in our past works on stochastic differential geometry applications to signal processing \cite{armstrongprojfilter}, see also  \cite{brigoieee}, \cite{brigobernoulli}.
Stratonovich calculus has become so dominant in stochastic differential geometry that some authors
have even asserted that stochastic differential geometry {\em requires} the use of Stratonovich calculus.
As It\^o's paper long ago demonstrated, this is not the case. Indeed under the formalism one can argue that it is still the It\^o calculus that ``does all the work'' (\cite{rogers}, Chapter V.30, p. 184).

From the point of view of this paper, we consider these two calculi as being simply different coordinate systems for the same underlying coordinate-free stochastic differential equation.
As such one should choose the most convenient coordinate system for the problem at hand. Stratonovich calculus has some clear advantages, most notably the transformation rule for vector fields is simpler than that for $2$-jets written using the standard representation. Another advantage of Stratonovich calculus is its stability with respect to regular noise in Wong-Zakai type results. It\^o calculus also has some clear advantages, it has simpler probabilistic properties stemming from the fact that the It\^o integral is non-anticipative and results in a martingale.

If one likes to think of differential geometry from an extrinsic perspective (i.e. if one doesn't like to think in terms of charts, but instead in terms of manifolds embedded in $\R^n$) then one can define an SDE on an embedded manifold by means of an SDE in $\R^n$ whose Stratonovich coefficients are vectors tangent to the manifold. This makes Stratonovich calculus a convenient way to introduce stochastic calculus on manifolds without needing to introduce the abstract definition of a manifold. However, this convenience does not imply that Stratonovich calculus is an essential tool for defining SDEs on manifolds. We note that our $2$-jet approach gives an equally simple approach to defining SDEs on embedded manifolds in terms of SDEs on $\R^n$. One simply requires that the $2$-jet at each point of the embedded manifold has a representative that lies entirely in that manifold. We could also say, in the embedded framework, that one simply requires that the curvature of the $2$-jet of our curve follows the curvature of the manifold.

In general, since Stratonovich calculus and It\^o calculus are just two coordinate systems one would expect to be able to work using either calculus interchangeably. The most important difference between Stratonovich calculus and It\^o calculus arises during 
the modelling process. It is when choosing what equation to write down in the first place that the choice between the calculi is most telling. Note that the modelling process is not a strictly mathematical process: it relies upon the modellers intuition.

From a modelling point of view It\^o calculus has a strong advantage for applications to subjects such as mathematical finance, where one considers decision makers who cannot use information about the future. The non-anticipative nature of the It\^o integral will make it easier to write down models for such problems using It\^o calculus.  On the other hand, in subjects such as physics or engineering, conservation laws, time-symmetry and the Wong-Zakai type convergence results mentioned in Appendix \ref{appendix:classicalFormulation} lead one to choose Stratonovich calculus. See for example the discussion in \cite{vankampen} on Langevin equations with exogenous noise, and the optimality criteria for projection on submanifolds discussed in \cite{armstrongBrigoOptimalProjectionArxiv}.

\section{Percentiles and fan diagrams}
\label{fanDiagramSection}

Most statistical properties of a distribution depend upon the coordinate system used. For example the definition of the mean of a process in $\R^n$ involves the vector space structure of $\R^n$. For this reason one would expect the trajectory of the mean of a process to depend upon the vector space structure. If one makes a non-linear transformation of $\R^n$ the trajectory of the mean changes. Indeed It\^o's lemma tells us that the trajectory does not even remain constant to first order under coordinate changes.

Another manifestation of the same phenomenon is the fact that given an $\R$ valued random variable $X$ and a non-linear function $f$, then $\E(f(X)) \neq f(\E(X))$. Again this arises because the definition of mean depends upon the vector space structure and $f$ may not respect this.

However, the definition of the $\alpha$-percentile depends only upon the ordering of $\R$ and not its vector space structure. As a result, for continuous monotonic $f$ and $X$ with connected state space, the median of $f(X)$ is equal to $f$ applied to the median of $X$. If $f$ is strictly increasing, the analogous result holds for the $\alpha$ percentile. If $f$ is decreasing, the $\alpha$ percentile of $f(X)$ is given by $f$ applied to the $1-\alpha$ percentile of $X$.

This has the implication that the trajectory of the $\alpha$-percentile of an $\R$ valued stochastic process is invariant under smooth monotonic coordinate changes of $\R$. In other words, percentiles have a coordinate free interpretation. The mean does not. This raises the question of how the trajectories of the percentiles can be related to the coefficients of the stochastic differential equation. We will now 
calculate this relationship.

\medskip

First we note that all smooth one dimensional Riemannian manifolds are isomorphic.
When interpreted in terms
of SDEs, this tells us that for any one dimensional SDE with non-vanishing diffusion coefficient (or volatility term) 
\begin{equation}
\ed X_t = a(X_t,t) \, \ed t + b(X_t,t) \ed W_t, \qquad X_0=x_0,
\label{oneDimensionalSDE}
\end{equation}
we can find a coordinate system with respect to which the volatility term is equal to one.
Such a transformation is known as a Lamperti transformation  (see for example \cite{pavliotis}) and is given by $Z_t = \phi(t,X_t)$ where $\phi(t,x)$ is a primitive integral of $1/b(x,t)$ with respect to $x$. 
Let 
\[ dZ_t = \alpha(Z_t,t) dt + dW_t, \  Z_0 = z_0 = \phi(0,x_0) \]
be the transformed equation.
Since one dimensional Riemannian
manifolds are translation invariant, we have a gauge freedom in defining a Lamperti transformation
determined by the base-point of the isomorphism.  Define a path
$z^0(t)$ by the ordinary differential equation
\begin{equation*}
\frac{ \ed z^0 }{\ed t} = \alpha(z^0(t),t), \ \ z^0(0) = z_0.
\end{equation*}
If we set $Y_t = Z_t - z^0(t)$ then $Y$ follows the SDE
\begin{equation}\label{eq:LampertiSDE} dY_t = \bar{a}(Y_t,t) dt  + dW_t, \ \ \ \bar{a}(y,t) = \alpha(y + z^0(t),t)   - \alpha(z^0(t),t), \ \  Y_0 = 0,
\end{equation}
and moreover the drift of the $Y$ SDE vanishes at $0$ for all $t$, $\bar{a}(0,t) = 0$. 

Summing up, if $b(x,t)$ is nowhere vanishing, there is a unique time dependent transformation
that maps the original SDE \eqref{oneDimensionalSDE} for $X$ into a SDE with constant diffusion coefficient, zero initial condition and zero drift at zero.  Actually, sufficient conditions for the Lamperti transformed SDE \eqref{eq:LampertiSDE} to have a unique strong solutions are more stringent. For example, in the autonomous case $a(x,t) = a(x)$ and $b(x,t)=b(x)$ one may need $a$ and $b$ bounded from below and above, with $b$ and its bounds strictly positive, with $a \in C^1$ and $b \in C^2$. We will assume in the following that sufficient conditions for the solution of the Lamperti transformed SDE to exist unique hold. 

Thus given a one dimensional SDE with non-vanishing volatility
we can always make a transformation such that (subject
to bounds) the following proposition applies (we denote $\bar{a}$ by $a$ for simplicity).

\begin{proposition}
\label{prop:beforeLamperti}
Let $a(x,t)$ be a bounded smooth function on $\R \times [0,+\infty)$ with $a(0,t)=0$. Define the operators $L$ and its adjoint $L^\ast$ as
\begin{eqnarray}
Lp &=&  \frac{1}{2} \frac{\partial^2 p }{\partial x^2 } + a(x,t) \frac{\partial p}{\partial x}
- \frac{\partial p}{\partial t}, \label{heatEquation} \\ \ L^\ast p &=&  \frac{1}{2} \frac{\partial^2 p }{\partial x^2 } -   \frac{\partial (a(x,t) p) }{\partial x}
- \frac{\partial p}{\partial t}.
\end{eqnarray}
Assume further that $a(x,t)$ has additional regularity required to ensure existence and uniqueness of 
a fundamental solution for the PDEs $L p =0$ and $L^\ast p =0$. 
Let $\Gamma(x,t;\xi,\tau)$ be the fundamental solution of $L p =0$. 
Then for $\lambda \in (0,1)$ and a fixed terminal time $T$, $\Gamma$ satisfies
\[
\Gamma(x,t; 0, 0) = \frac{1}{\sqrt{2 \pi t}} \exp\left( -\frac{x^2}{2 t} \right)
+ O\left( \frac{t + x^2}{\sqrt t} \exp\left( -\frac{\lambda x^2}{2 t} \right)  \right)
\]
on $\R \times [0,T]$, and the fundamental solution of $L^\ast p =0$ satisfies
\[ \Gamma^*(0,0;x,t) = \Gamma(x,t;0,0)\]
 (see \cite{friedman} for the definition of a fundamental solution).
\end{proposition}
\begin{proof}
Equation $L^\ast p =0$ as the evolution for the density of a solution of a SDE is studied for example in Friedman's SDE book \cite{friedmansdes}, see Eq. 5.28 in Chapter 6: by Theorem 4.7 in Ch 6  the fundamental solution of $L^\ast p=0$ is the same as the fundamental solution of $L p =0$. We will thus focus on $L p=0$, for which one can refer to Ch 1 Section 6 of Friedman's parabolic PDEs book \cite{friedman}.  
When $a$ is identically zero, the fundamental solution of $L p =0$ in \eqref{heatEquation} is given by
\[
Z(x,t;\xi,\tau) = \frac{1}{\sqrt{2 \pi (t - \tau)}} \exp\left( -\frac{(x-\xi)^2}{2 (t- \tau)} \right).
\]
By (\cite{friedman} Ch 1, Theorem 10, p.\ 23), the fundamental solution of \eqref{heatEquation} is given by
\begin{equation}
\Gamma(x,t;\xi,\tau) = Z(x,t;\xi,\tau) + \int_{\tau}^{t}\int_{\R} Z(x,t; \eta, \sigma)
\Phi(\eta, \sigma; \xi,\tau) \, \ed \eta \, \ed \sigma
\end{equation}
where
\begin{equation*}
\Phi(x,t;\xi,\tau) = \sum_{i=1}^{\infty} \Phi_i(x,t; \xi, \tau)
\end{equation*}
where we define $\Phi_1$=$LZ$ and for $i \ge 1$ we define
\begin{equation*}
\Phi_{i+1}(x,t;\xi,\tau) = \int_{\tau}^{t} \int_{\R} L\ Z(x,t;y,\sigma) \Phi_i(y,\sigma,\xi,\tau)
\, \ed y \, \ed \sigma.
\end{equation*}
Let us define:
\[
\Gamma_i(x,t;\xi,\tau) := \int_\tau^t \int_{\R} Z(x,t;\eta,\theta) \Phi_i( \eta, \theta; \xi,\tau) \, \ed \eta \, \ed \theta.
\]

First we bound  $\Gamma_1$.
\begin{equation*}
\begin{split}
\Gamma_1(x,t;\xi,\tau) 
&= \int_{\tau}^{t}\int_{\R} Z(x,t; \eta, \sigma)
LZ(\eta, \sigma; \xi,\tau) \, \ed \eta \, \ed \sigma \\
&= \int_{\tau}^{t}\int_{\R} Z(x,t; \eta, \sigma)
a(\eta,\sigma) \frac{\partial}{\partial \eta} Z(\eta, \sigma; \xi,\tau) \, \ed \eta \, \ed \sigma. \\
\end{split}
\end{equation*}
The final line follows because $Z$ is the fundamental solution of \eqref{heatEquation} when $a$ is
equal to zero and so the parts of $\Gamma_1$ that don't involve $a$ vanish.
Our assumptions ensuring Lipschitz continuity of $a(x,t)$ uniformly in $t$ and the fact that $a(0,t)=0$ for all $t$ (see \eqref{eq:LampertiSDE}) tell us that there is some constant with $|a(x, t)|<C|x|$ for all $x$
and $t \in [0,T]$. So
\[
\left|a(\eta,\sigma) \frac{\partial}{\partial \eta} Z(\eta, \sigma; 0,0) \right|
= \left|-\frac{1}{\sqrt{2 \pi \sigma}} a(\eta, \sigma) \eta \exp\left( -\frac{\eta^2}{2 \sigma} \right)\right|
< \frac{C}{\sqrt{2 \pi \sigma}} \eta^2 \exp\left( -\frac{\eta^2}{2\sigma} \right).
\]
We deduce that
\begin{equation*}
|\Gamma_1(x,t;0,0)| \leq C \int_0^t \int_\R Z(x,t;\eta,\sigma) \eta^2
\frac{1}{\sqrt{2 \pi \sigma}} \exp\left( -\frac{\eta^2}{2\sigma} \right) \ed \eta \, \ed \sigma
=  C \frac{e^{-\frac{x^2}{2 t}} \left(t+x^2\right)}{2 \sqrt{2 \pi t }}.
\end{equation*}
The final step follows simply by the routine evaluation of the integral. 
To do this we use estimates which were used in \cite{friedman} to prove that our expression
for $\Gamma$ exists and is a fundamental solution. First, we have for any $\lambda \in (0,1)$ that
\begin{equation*}
|Z(x,t,\xi,\tau)| \leq \frac{1}{(t-\tau)^\frac{1}{2}}
\exp\left( -\frac{\lambda|x-\xi|^2}{2(t - \tau)} \right).
\end{equation*}
From (\cite{friedman} Ch 1, 4.14,  p.\ 16) we have that there exist positive constants $H$ and $H_0$ such that for all $\lambda \in (0,1)$ and $m\geq2$
\begin{equation*}
|\Phi_m(x,t; \xi, \tau)| \leq \frac{H_0 H^m}{\Gamma_E(m)} (t-\tau)^{m - \frac{3}{2}} \exp\left( -\frac{\lambda|x-\xi|^2}{2(t-\tau)} \right)
\end{equation*}
where $\Gamma_E$ is the gamma function. From (\cite{friedman} Ch 1, Lemma 3 p.\ 15) we have that if
$r<\frac{3}{2}$ and $s<\frac{3}{2}$ then
\begin{multline*}
\int_{\sigma}^{\tau}\int_\R
(t-\tau)^{-r} \exp\left( -\frac{h(x-\xi)^2}{2(t-\tau)}  \right)
(\tau-\sigma)^{-s} \exp\left( -\frac{h(x-\xi)^2}{2(t-\sigma)} \right) \, \ed \xi \, \ed \sigma \\
= \sqrt{\frac{2 \pi}{h}}
B\left( \frac{3}{2}-r,\frac{3}{2}-s\right)(t-\sigma)^{\frac{3}{2}-r-s}
\exp\left( - \frac{h(x-y)^2}{2(t - \sigma)} \right)
\end{multline*}
where $B$ is the beta function. Taking $r=1/2$ and $s=\frac{3}{2}-m$ we obtain the estimate
\[
\begin{split}
|\Gamma_m(x,t; \xi, \tau)|
& \leq
\frac{H_0 H^m}{\Gamma_E(m)}
\sqrt{\frac{2 \pi}{\lambda}}
B( 1,m)(t-\tau)^{-\frac{1}{2}+m}
\exp\left( - \frac{\lambda(x-\xi)^2}{2(t - \tau)} \right) \\
&= \frac{H_0 H^m}{m!} \sqrt{\frac{2 \pi}{\lambda}}
(t-\tau)^{-\frac{1}{2}+m}
\exp\left( - \frac{\lambda(x-\xi)^2}{2(t - \tau)} \right)
\end{split}.
\]
So
\[ 
\begin{split}
\sum_{k=2}^\infty
|\Gamma_m(x,t; 0, 0)| &\leq \sum_{m=2}^\infty H_0 \frac{H^m}{m!} t^{-\frac{1}{2}+m}
\exp \left(-\frac{\lambda x^2}{2 t} \right) \\
&=  \sqrt{\frac{2 \pi}{\lambda}} H_0 H^2 \exp(H t )  (t)^{3/2} \exp \left(-\frac{\lambda x^2}{2 t} \right)
\end{split}
\]
The result follows if we conclude by invoking Theorem 15, Ch 1, p. 28 of \cite{friedman}, or Theorem 4.7 in Ch 6  of \cite{friedmansdes}.
\end{proof}

\begin{theorem}
For sufficiently small $t$, the $\alpha$-th percentile of the solutions to
\eqref{oneDimensionalSDE} is given by
\begin{equation}\label{eq:percentilejetsalpha}
x_0 + b_0 \sqrt{t} \Phi^{-1}(\alpha) + \left(a_0 - \frac{1}{2} b_0 b_0^\prime (1- \Phi^{-1}(\alpha)^2 )\right)t + O(t^{3/2})
\end{equation}
 so long as the coefficients of \eqref{oneDimensionalSDE} are smooth, the diffusion coefficient $b$ never vanishes, and  sufficient conditions for the Lamperti transformed SDE and for $L^\ast p=0$ to have a unique regular solution hold.  
In this formula $a_0$ and $b_0$ denote the values of $a(x_0,0)$ and $b(x_0,0)$ respectively. 
In particular, the median process is a straight line up to $O(t^\frac{3}{2})$ with tangent given by the drift of the Stratonovich version of the It\^o SDE \eqref{oneDimensionalSDE}. The $\Phi(1)$ and $\Phi(-1)$ percentiles correspond up to $O(t^\frac{3}{2})$ to the curves $\gamma_{X_0}(\pm \sqrt{t})$ where $\gamma_{X_0}$ is any representative of the $2$-jet that defines the SDE in It\^o form. 
\end{theorem}
\begin{proof}
We first apply a Lamperti transformation so that the conditions of Proposition \ref{prop:beforeLamperti}
apply. Let write $y$ for the coordinates after applying the Lamperti transformation and let
us write $\rho$ for the 1-form representing the probability measure. By Proposition \ref{prop:beforeLamperti}
\[
\rho = \left[ \frac{1}{\sqrt{2 \pi t}} \exp\left(-\frac{y^2}{2 t}\right) + O\left( \frac{t + y^2}{\sqrt t} \exp\left( -\frac{\lambda y^2}{2 t} \right) \right) \right] \ed y.
\]
We introduce a new coordinate $z = \frac{y}{\sqrt{t}}$ so that the $\rho$ can be written
\[
\rho = \left[ \frac{1}{\sqrt{2 \pi}} \exp\left(-\frac{z^2}{2}\right) + O\left( t(1+z^2) \exp\left( -\frac{\lambda z^2}{2} \right) \right) \right] \ed z.
\]
Integrating this, we see that for sufficiently small $\epsilon$ we have a uniform estimate
\[
\int_{-\infty}^{\Phi^{-1}(\alpha+\epsilon)} \rho = \alpha + \epsilon + O( t ).
\]
It follows that the $\alpha$-th percentile in $z$ coordinates is $\Phi^{-1}(\alpha) + O(t)$.
Hence in $y$ coordinates it is $\phi^{-1}(\alpha)\sqrt{t} + O(t^{3/2})$.

Let $g$ denote the inverse of the Lamperti transformation that we have taken. Let us write
the Taylor series expansion for $g$ around (0,0) as
\begin{equation}\label{eq:gtaylor}
g(y,t)=x_0 + g_y y + \frac{1}{2} g_{yy} y^2 + g_t t + O(y^3 + t^2) .
\end{equation}
By construction the SDE in $y$ coordinates is of the form
\[
\ed Y_t = \overline{a} \ed t + \ed W_t
\]
with $\overline{a}(0,t)=0$. When we apply $g$ to this equation we will get equation
\eqref{oneDimensionalSDE}. Using It\^o's lemma we deduce that:
\[
b = g_y, \quad
a_0 = g_t + \frac{1}{2} g_{yy}.
\]
The first of these equations holds everywhere, the second uses the vanishing of $\overline{a}$
at $0$.
To avoid notational ambiguity in partial differentiation, we will temporarily write
$s$ for the time coordinate when paired with $x$ and $t$ for the time coordinate when
paired with $y$. Thus we are making the coordinate transformation
\[
x = g(y,t), \quad s = t.
\] 
We then have that
\[
\frac{\partial b}{\partial x} = \frac{\partial b}{\partial y} \frac{ \partial y}{\partial x}
+ \frac{\partial b}{\partial s} \frac{\partial s}{\partial x}
= g_{yy} \frac{1}{g_y}.
\]
So that we have that at $0$
\[
g_y = b_0, \quad
g_{yy} = b_0 b_0^\prime, \quad
g_t = \left( a_0 - \frac{1}{2} b_0 b_0^\prime \right) .
\]
Substituting these formulae 
into the Taylor series expansion \eqref{eq:gtaylor} for $g$ yields
\[
g(y,t) = x_0 + b_0 y + \frac{1}{2} b_0 b_0^\prime y^2 + (a_0 - \frac{1}{2} b_0 b_0^\prime) t + O(y^3+t^{2}).
\]
Substitute $y=\Phi^{-1}(\alpha)\sqrt{t}$ to get
\[
x_0 + b_0 \Phi^{-1}(\alpha)\sqrt{t} + \frac{1}{2} b_0 b_0^\prime \Phi^{-1}(\alpha)^2 t + (a_0 - \frac{1}{2} b_0 b_0^\prime) t
+ O(t^{3/2}).
\]
This simplifies to \eqref{eq:percentilejetsalpha}.

\end{proof}

The theorem above has given us the median as a special case, and a link between the median and the Stratonovich version of the SDE. 
By contrast the mean process has tangent given by the drift of the It\^o SDE as the It\^o integral is a martingale.

For completeness, besides the mean and the median, we wish to consider the mode. We claim that under the same conditions as the theorem above, there are paths $m^u(t)$ and $m^l(t)$ both satisfying
\[
\begin{split}
m^u &= x_0 + a(x_0,0)t - \frac{3}{2} b(x_0,0) b^\prime(x_0,0) t + O(t^\frac{3}{2}) \\
m^l &= x_0 + a(x_0,0)t - \frac{3}{2} b(x_0,0) b^\prime(x_0,0) t + O(t^\frac{3}{2}) 
\end{split}
\]
such that for sufficiently small $t$ there exists a mode lying in $[m^l(t), m^u(t)]$
This relationship between the mean, median and mode is approximately seen in many general probability distributions as was observed qualitatively by Pearson \cite{pearson}.
 
This result gives an alternative way of plotting the two jet that defines a one dimensional stochastic differential equation in terms of fan diagrams.
A {\em fan diagram} is a standard tool in econometrics for illustrating the predictions
of a model. In Figure \ref{fig:genericfandiagram} we have plotted a fan diagram for a stock price modelled by geometric Brownian motion. The negative times in the plot show historical values for the stock price.
For positive times, we plot a single random realization of geometric Brownian motion together with
two percentiles that indicate the range of values attained by other realizations. We have
chosen to plot the percentiles $\Phi(1) \approx 84\%$ and $\Phi(-1) \approx 16\%$.

We can use the result above to plot an SDE by drawing an infinitesimal fan-
diagram at each point.
At each point $x \in \R$ one plots the curve $(t, \gamma_x(\pm \sqrt{t}))$.
One interprets this diagram as an infinitesimal fan-diagram showing the $\Phi(1)$ and $\Phi(-1)$ percentiles. Such a plot is shown for the process $\ed S_t= \sigma S_t \ed W_t$ in the left hand panel of \ref{fig:fanDiagram}.

\begin{figure}[ht]
\centering
\begin{minipage}[t]{0.45\linewidth}
\includegraphics[width=0.9\linewidth]{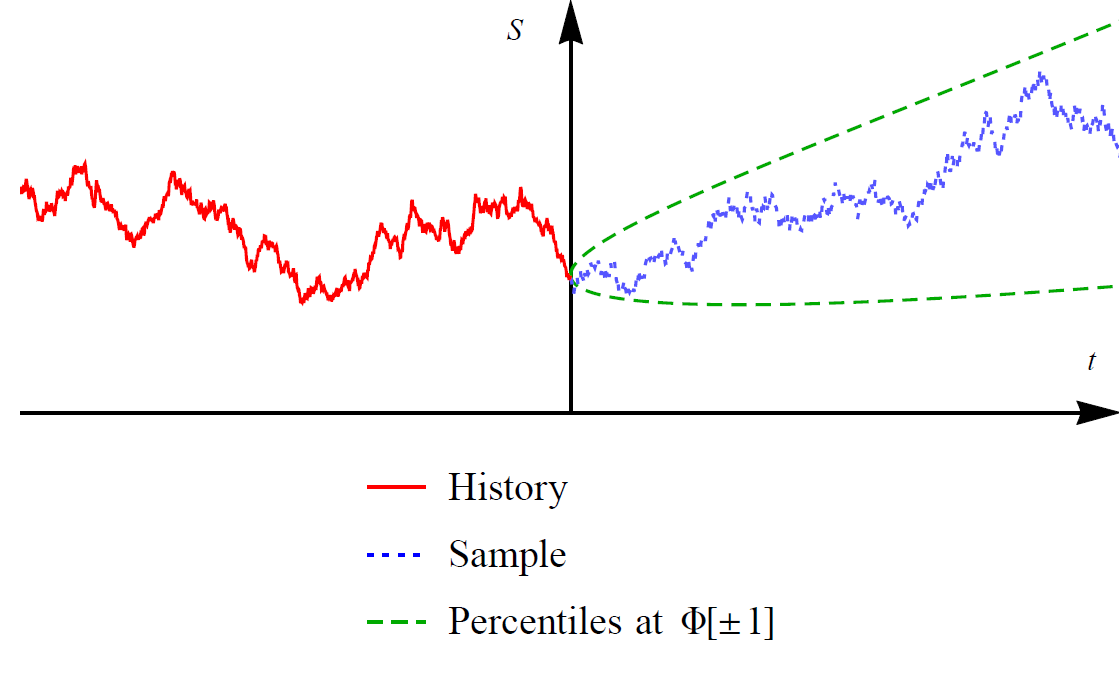}
\caption{A fan diagram for a stock price, $S$ modelled using geometric Brownian motion}
\label{fig:genericfandiagram}
\end{minipage}
\hfill
\begin{minipage}[t]{0.45\linewidth}
\includegraphics[width=0.4\linewidth]{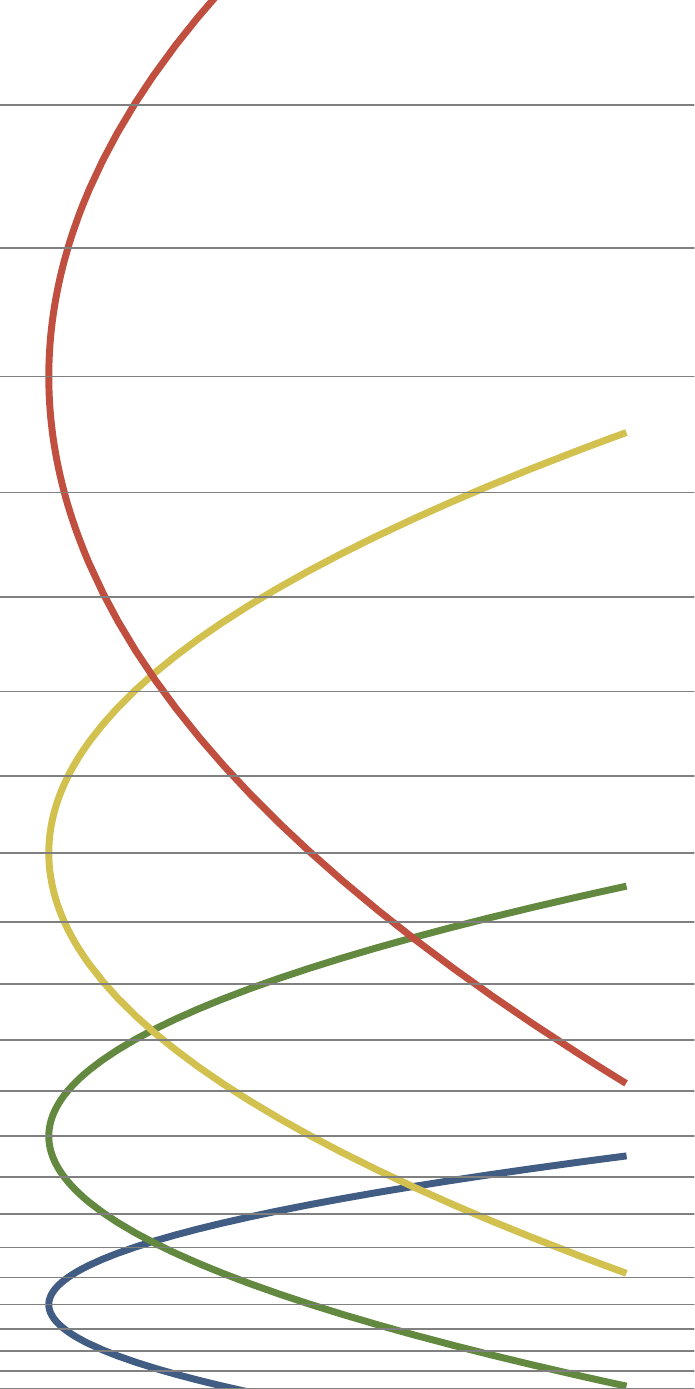}
\hfill
\includegraphics[width=0.4\linewidth]{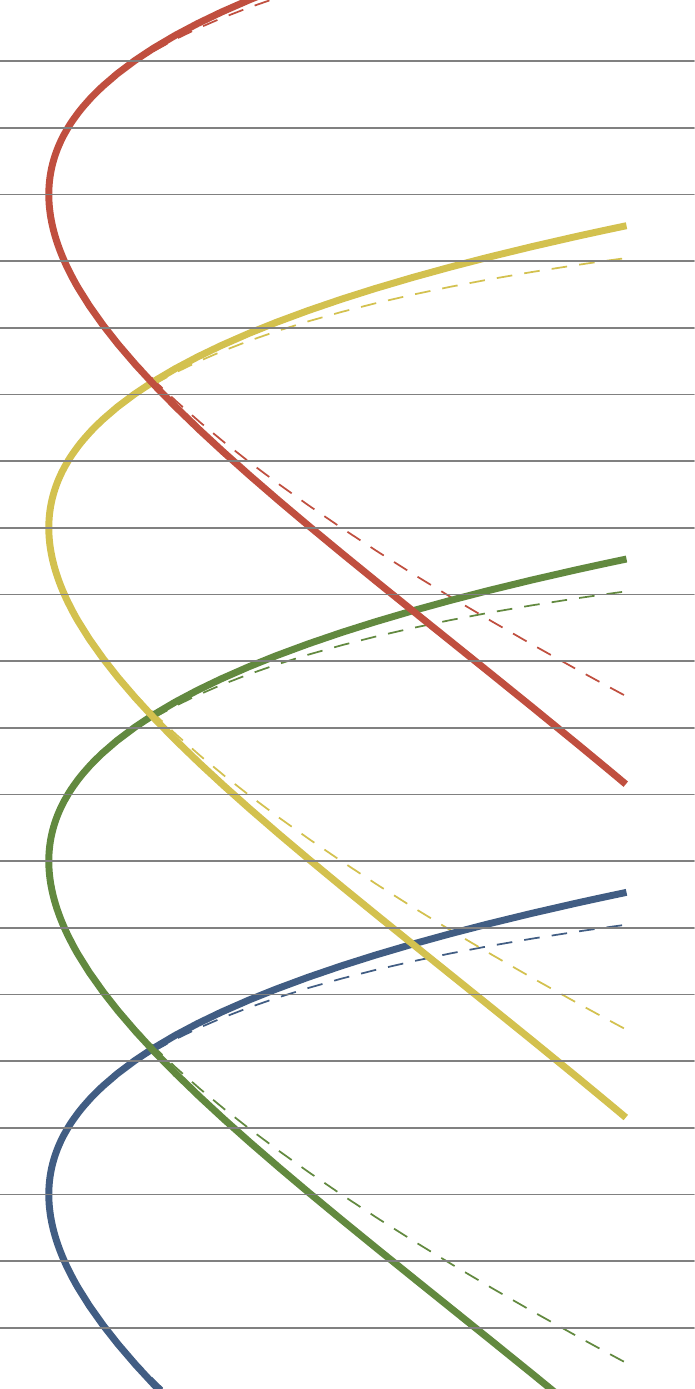}
\caption{A fan diagram of $\ed S_t = S_t \sigma_t \ed W_t$ (left). The result of applying $(x,y)\to(x,\log y)$ to this fan diagram (right, solid line). A fan diagram of equation \eqref{logProcess} (right, dashed line). }
\label{fig:fanDiagram}
\end{minipage}
\end{figure}

In the right hand panel of \ref{fig:fanDiagram} we show how this plot transforms when
one sends $(t,S_t)$ to $(t, \log(S_t))$. This is illustrated with solid lines. We also
use dashed lines to plot the corresponding diagram for the equation arising from It\^o's lemma, namely
\begin{equation}
\ed (\log(S_t)) = -\frac{1}{2}\sigma^2 \ed t + \sigma \, \ed W_t.
\label{logProcess}
\end{equation}

It is interesting to notice that one can clearly see the drift term in the right hand side
of \ref{fig:fanDiagram}. Notice also that this drift arises because the spacing between
the grid lines on the left hand side of \ref{fig:fanDiagram} increases as one moves up
the page whereas the corresponding grid lines after transformation are evenly spaced. This
can be interpreted as a visual demonstration that the It\^o term in the transformation rule for SDEs
is determined by the second derivative of the transformation.

\section{Conclusions and further work}
In this first work we showed how It\^o SDEs driven by Brownian motion could be understood in terms of 2-jets of maps.
 In further work we will study more in detail the relationship between the jets approach and Schwartz morphism based on second order tangent vectors and co-vectors (as in the Schwartz - Meyer theory explained in  Emery's work \cite{emery}). 
 We will study the relationship with Belopolskaja and Dalecky's It\^o bundle \cite{belopolskaja}, see also \cite{gliklikh} and the appendix in \cite{brzezniak}.
 We will explore the jet approach in connection with projection method for dimensionality reduction and approximation of SDEs \cite{armstrongBrigoOptimalProjectionArxiv,armstrongbrigoicms}. We will explore practical applications of our apporach to simulating SDEs on manifolds \cite{armstrongMijatovic}. We plan to investigate the use of jets in rough paths theory \cite{frizhairer}.
 
\section*{Acknowledgments} We are grateful to David K. Elworthy for important feedback, suggestions, 
correspondence and for suggesting links with other approaches such as Schwartz 
morphism and the It\^o bundle. We are indebted to Jeremiah Birrell for many important clarifications and corrections. DB also benefited from correspondence with Michel Emery on the Schwartz morphism approach, and from further correspondence with Giovanni Paolinelli and Giovanni Pistone. This led to Appendix B, which will be expanded in our future work. 

\newpage

\section*{Appendices} 

\begin{appendix}
	
\ifarxiv{
\section{Classical formulation}
\label{appendix:classicalFormulation}
This section is a relatively informal introduction to stochastic differential equations (SDEs) in a given coordinate system. The reader who is familiar with SDEs in both It\^o and Stratonovich form may skip this section.

An SDE is typically written as 
\[ dX_t = a(X_t) dt + b(X_t) dW_t , \ \ X_0 .\]
This is a  It\^o stochastic differential equation. Informally, in the one-dimensional case, we can interpret the first coefficient $a(X_t) dt$ as the local mean for $dX_t$ given past and present information up to $t$, and $b(X_t)^2 dt$ as the local variance of $dX_t$ given past and present information.     $X_0$ is the initial condition at time $0$. The input $W$ models random noise that is modelled by Brownian motion, whose formal derivative is a  model for white noise. Brownian motion is a stochastic process with continuous paths and stationary independent increments. $\delta W_t = W_{t+\delta t}-W_t$ is independent of $W_t$ and is normally distributed with zero mean and variance $\delta t$. 

The problem with properly defining the above SDE is that $W$'s paths have unbounded variation, and are nowhere differentiable with probability one.  So the above cannot be interpreted as a pathwise differential equation directly. One then writes it as an integral equation
\begin{equation}\label{integraleq} X_t = X_0 + \int_0^t a(X_s) ds + \int_0^t b(X_s) dW_s  .
\end{equation}
Now the matter is defining the stochastic integral driven by $dW$, ie the last term in the right hand side. 
Since $W$ has unbounded variation, we cannot define this as an ordinary Stiltjes integral on the paths.
Traditionally, two main definitions of stochastic integrals are available, given here as generalizations of Riemann-Stiltjes sums for comparison convenience, with convergence in mean square ($L^2(\mathbb{P})$, where $\Pb$ is the probability measure in the probability space where the SDE is defined). The two choices are whether one takes the initial point or the mid point in defining the Stiltjies sums, leading respectively to: 
\[ \int_0^T b(X_s) dW_s = \lim_{n} \sum_{i=1}^n b(X_{{t_i}}) (W_{t_{i+1}}-W_{t_i})  \ \ \mbox{(It\^o)} \]
% \pause
\[ \hspace{-0.5cm} \int_0^T  b(X_s) \circ dW_s = \lim_{n} \sum_{i=1}^n {b\left(X_{\frac{t_i+t_{i+1}}{2}}\right)}  (W_{t_{i+1}}-W_{t_i}) \ \  \mbox{(Stratonovich)}. \]
The Stratonovich integral has also a more general definition with  \[ [b(X(t_i))+ b(X(t_{i+1}))]/2\]  in front of $(W_{t_{i+1}}-W_{t_i}$). In the above limits it  is understood that as $n$ tends to infinity the mesh size of the partition  \[\{[0,t_1),[t_1,t_2), \ldots, [t_{n-1},t_n=T]\}\] of $[0,T]$ tends to 0. 
% \pause
The Stratonovich integral, also known as Fisk (\cite{fisk})-Stratonovich (\cite{strato}) (-McShane \cite{mcshane})  integral, is said to look into the future, since the integrand is evaluated after the initial time $t_i$ of each increment, whereas the It\^o integrand is evaluated at the initial time of each increment, at $t_i$, and ``does not look into the future". 

The It\^o integral is a martingale, an important type of stochastic process in probability theory. A consequence of the martingale property is that we can indeed interpret the drift term $a(x)dt$ as a local mean for $dX$. This does not hold with the Stratonovich integral, although we will argue that for Stratonovich integrals there is a link between the drift and the median. More generally, the It\^o integral has a number of good probabilistic properties. Its main problem is that it does not satisfy the chain rule. Under a change of variables driven by a transformation $f$ from $\R^n$ to $\R$, the SDE changes according to It\^o's lemma, leading to
\[ d f(X_t) = ((\partial_x f) (X_t))^T\ dX_t + \frac{1}{2} (d X_t)^T (\partial^2_{xx} f(X_t)) ( dX_t) \]
where $\partial_x$ and $\partial^2_{xx}$ denote respectively the gradient and Hessian of $f$, and upper $T$ denotes transposition.
A more precise way to write the same equation is by explicitly referring to components, namely
\[ d f(X_t) = ((\partial_i f) (X_t))^T\ dX^i_t +  \frac{1}{2} ^T (\partial^2_{ij} f(X_t)) d[ X^i_t, X^j_t]\]
where $[\cdot,\cdot]$ denotes the quadratic covariation. 
The short hand for the Stratonovich SDE is written through the It\^o's circle ``$\circ$" notation: If we use the Stratonovich integral in \eqref{integraleq}, then the abbreviated notation for the related SDE is
\[ d X_t = a(X_t) dt + b(X_t) \circ dW_t.\]
While lacking the good probabilistic properties of the It\^o integral, and in particular not being a martingale, the Stratonovich integral has an important property: the related SDE satisfies the chain rule under a change of variables, namely
\[ d f(X_t) = ((\partial_x f) (X_t))^T\circ dX_t . \]
Given that the chain rule holds, the coefficients of SDEs in Stratonovich form behave like vector fields under changes of coordinates. This is why the Stratonovich integral is good for coordinate-free SDEs on manifolds and stochastic differential geometry more generally.
One more important advantage of the Stratonovich integral is its convergence under smooth noise. If we take $t$-$C^1$ processes $t \mapsto W^{(n)}(t)$ such that $W^{(n)} \rightarrow W$ with probability 1, uniformly in $t$-bounded intervals, then the following Wong--Zakai type of result holds.

\newpage

\[\mbox{The pathwise regular solution of} \ dX^{(n)}_t = a(X^{(n)}_t) dt + b(X^{(n)}_t) dW^{(n)}_t \]
\[ \mbox{converges to the solution of} \ \  
dX_t = a(X_t) dt + b(X_t)\circ dW_t \]
(see \cite{oksendal} for a relatively informal account).
So if we approximate the rough noise $W$ with more regular noise converging to $W$ and take the limit, we obtain the Stratonovich-integral based solution and not the It\^o one. 

Summing up, we could say that It\^o SDEs are good probabilistically, while Stratonovich SDEs are good geometrically.
In this paper we ask ourselves: can we make the It\^o-integral based SDEs coordinate free and good for geometry, so as to have good  probabilistic and geometric properties together? One of the aims of this paper is to show that the answer is affirmative, and that the notion of jet plays a key role in the answer. In closing this appendix, we would like to shortly mention the It\^o - Stratonovich transformation. 
Indeed, under assumptions guaranteeing that both equations have solutions that are regular enough, there is a simple rule to re-write an It\^o SDE into a Stratonovich SDE admitting the same solution (and vice versa).
This is known as It\^o - Stratonovich transformation. We will give a coordinate-free interpretation of this transformation later on in the paper. The transformation works as follows:
\[ d X_t = a(X_t) dt + b_i(X_t) dW^i_t \ \ \rightarrow \ \ d X_t = \tilde{a}(X_t)dt  + b_i(X_t)\circ dW^i_t\]
\[ \tilde{a}_i = a_i - \frac{1}{2} \sum_{k=1}^{\dim W} \sum_{h=1}^{\dim X} \frac{\partial b_{i,k}}{\partial x_k} b_{h,k} \]
where in the $dX$ equations we are using Einstein summation convention. The two equations above have the same solution. 
So if we need geometry and Ito, why not do the following:
\begin{itemize}
\item Given the initial Ito SDE, transform it in Stratonovich form with the rule above.
The solution will be the same.
\item Work (with geometry) using the Stratonovich SDE.
\item Once you are done, convert the equation back into Ito form.
\end{itemize} 
The ``work with geometry" at the moment is a little vague. Even so, we may already point out that the It\^o Stratonovich transformation does not commute with  operations related to methods for projection on submanifolds.   See for example the exponential-families assumed-density-filters (equivalent to Hellinger metric projection for Stratonovich SDEs) in stochastic filtering \cite{brigobernoulli}, or more importantly the optimal projection of SDEs on submanifolds \cite{armstrongBrigoOptimalProjectionArxiv}, where it is shown that the jet formulation of It\^o SDEs introduced in this paper leads to a projection that has better optimality properties than the straightforward projection of the Stratonovich SDE. In that case working directly with the It\^o version via jets or switching to Stratonovich and projecting the Stratonovich version yield different results.
%In these cases
%\begin{itemize}
%\item Transform  full It\^o SPDE for optimal filter in Stratonovich form;
%\item Apply the assumed density/``expected sufficient statistics matching" approximation to obtain a finite dimensional approximating Stratonovich SDE in the chosen exponential family;
%\item Convert the approximating  Stratonovich SDE into It\^o form; 
%\end{itemize}
%produces a different  It\^o SDE with respect to the direct method:
%% \pause
%\begin{itemize}
%\item 
%Apply  assumed density approximation directly to the It\^o full SPDE.
%\end{itemize}
%
%Along similar lines, we point out that many important notions in probability theory, starting from the expectation, are not coordinate free, in that they are not covariant. 

It is therefore important to make a choice about the stochastic integral to be employed before facing particular problems.

\section{Equivalent formulations}
\label{appendix:equivalentFormulations}
In this section we briefly connect earlier approaches to SDEs on manifolds with our jets representation. We start from Schwartz morphism, the approach originated by Schwartz and Meyer, see for example Emery \cite{emery,emeryhal}. We will give a rather informal summary of this approach. This approach is based on second order tangent vectors \cite{emery}, also known as diffusors \cite{emeryhal}. These are defined as second order differential operators without constant term. These are applied to real functions defined on the manifold.

Before describing diffusors a little more in detail and connecting them with 2-jets, let's recall the definition of tangent vectors at a point $p$ of a manifold $M$. A possible definition of tangent vector is a first order linear differential operator with no constant term acting on real functions defined on the manifold $M$. So for example $L_p$ is an operator such that, for $f: M \rightarrow \R$, $L_p f$ is a real number. If $\gamma$ is a curve on the manifold with $\gamma(0)=p$, then $L^\gamma_p f := (f\circ \gamma)'(0)$ is a differential operator tangent vector associated to $\gamma$. One can prove that all $L_p$ are like this, i.e. they all come from the velocity of some curve $\gamma$. 

We can compute the expression for the tangent vector in a coordinate system $x$, where we can express both functions $f$ and curves $\gamma$ on the manifold via their coordinate versions $\tilde{f}$ and $\tilde{\gamma}$, see Fig. \ref{fig:charts}. We obtain
\[ L^{\gamma}_p f = (f \circ \gamma)'(0) = (\tilde{f} \circ {\tilde{\gamma}})'(0) = 
\frac{\partial \tilde{f}}{\partial \tilde x^k}(\tilde{\gamma}(0)) \tilde{\gamma}'_k(0) = L^k \frac{\partial \tilde{f}}{\partial x^k}(x(p))\]
from which we have $L^k$ in coordinates for tangent vectors.

\bigskip

\begin{figure}[h!]
\begin{center}
\includegraphics{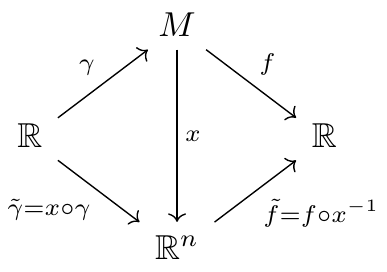}
\end{center}
\caption{Coordinate versions of functions and curves on the manifold}\label{fig:charts}
\end{figure}

The space of all tangent vectors at $p$ to $M$ is denoted $T_pM$. The tangent bundle $TM$ is the union of all tangent spaces for $p$ ranging the manifold, with the unique topology and smooth structure that make the trivializations diffeomorphisms. 
A tangent vector field is a smooth function $V: M \rightarrow TM$, or a section of $TM$ such that for $p \in M$ then $V(p) \in T_pM$. 
%A co-tangent vector (briefly co-vector) or differential form of degree 1 at $p$ is an element of the dual space of the tangent space,
%$T^\ast_p M$, thus it is a linear operator that given a tangent vector at $p$ gives back a number. 
%Co-vectors at $p$ are typically denoted $d f(p)$, since they are associated with a function $f$ on the manifold, and are such that $df(p)(\tilde{\gamma}) = df(p)(L^\gamma_p) =(L^\gamma_p) f = (f \circ \gamma)'(0)$.
%
%The co-vectors associated to $\partial / \partial x^i|_p = \tilde{\gamma}_i(p)$ are denoted by $d x^i(p)$. The $d x^i(p)$ form a basis of $T^\ast_p M$. 
%
%The set $T^\ast M = \cup_p T^\ast_p M$ is the co-tangent bundle. 
%A co-tangent vector field (briefly co-vector field or just form) is a smooth map $\alpha: M \rightarrow T^\ast M$ such that $\alpha(p) \in T^\ast_p M$.  Note that even if co-vectors at $p$ or forms at $p$ can be written all as $df(p)$, not all forms or co-vector fields can be written as $df$. 

We define a second order tangent vector or diffusor at  $p\in M$ as a second order linear differential operator with no constant term, transforming smooth real functions on $M$ in numbers.
In coordinates $x$, one has 
\begin{equation}\label{eq:diffusorL} L_p f = L^{ij} \frac{\partial^2 f}{\partial x^i \partial x^j}\bigg{|}_p + L^k \frac{\partial f}{\partial x^k}\bigg{|}_p . 
\end{equation}
First notice that when the real numbers $L^{ij}=0$ then the diffusor at $p$ is a tangent vector at $p$. So tangent vectors are special cases of diffusors.

Due to the commutative property 
\[ \frac{\partial^2 f}{\partial x^i \partial x^j}\bigg{|}_p = \frac{\partial^2 f}{\partial x^j \partial x^i}\bigg{|}_p , \ \ i \neq j\]
we have that a basis for diffusors at $p$ is for example
\begin{equation}\label{eq:diffusorsbasis} \left\{ \frac{\partial^2 }{\partial x^i \partial x^j}\bigg{|}_p, \ \frac{\partial }{\partial x^k}\bigg{|}_p, \ \ i,j,k=1,\ldots,n, \ \ \underline{i\le j} \right\} .
\end{equation}
Diffusors are related to acceleration of curves. 
We can compute the second order operator generated by the acceleration of a curve like this:
\[ L^{\gamma}_p f = (f \circ \gamma)''(0) = (\tilde{f} \circ {\tilde{\gamma}})''(0). \]
In coordinates this reads
\begin{equation}\label{eq:diffusor} L^{\gamma}_p f = (\tilde{f} \circ \tilde{\gamma})''(0) =  \frac{\partial^2 \tilde{f}}{\partial \tilde x^k \partial \tilde x^h}(\tilde{\gamma}(0)) \tilde{\gamma}'_k(0)\tilde{\gamma}'_h(0) + \frac{\partial \tilde{f}}{\partial \tilde x^k}(\tilde{\gamma}(0)) \tilde{\gamma}''_k(0).  
\end{equation}
In the basis \eqref{eq:diffusorsbasis}, the diffusor \eqref{eq:diffusorL} would present us with the cross terms $L^{ij}+L^{ji}$ when $i<j$. We need to decide how to split the sum into the two components. The most natural assumption is to assume symmetry and take $L^{ij}=L^{ji}$. We get then the following coordinate expression for the acceleration diffusor:
\[ L^{ij} = \tilde{\gamma}'_i(0)\tilde{\gamma}'_j(0), \ \ L^k = \tilde{\gamma}''_k(0), \ \ i\le j, \ i,j,k=1,\ldots,n.\]
We notice that the rank of the symmetric matrix $L^{ij}$ is
one for the diffusor associated with the acceleration of a curve. It follows
that diffusors with higher rank $L^{ij}$ cannot be represented directly as
pure-accelerations diffusors, but can be represented as a linear combination
of pure-acceleration diffusors.
The set of all diffusors at $p$ in $M$ is the second order tangent space, and is denoted by $\tau_p M$ \cite{emery} or $\Ti_p M$ \cite{emeryhal}.

The correspondence we have described in this paper between the two jet of a curve and the associated differential operator, namely \begin{equation*}
{\cal L}_{\gamma_x}(f):=\frac{1}{2}(f \circ \gamma_x)^{\prime \prime}(0)
\end{equation*}
is precisely the same as the above correspondence between the acceleration of a curve and a diffusor described by Emery. Our
contribution is to give a simple and intuitive interpretation of each curve
in terms of a numerical scheme, to extend the interpretation to diffusors
where $L^{ij}$ has rank greater than one and to provide an interpretation of Stratonovich calculus in terms of curves.

%
%We can compare the action of a pure-acceleration diffusor to a 2-jet. Consider the two jet of the function $t \mapsto (\tilde{f}\circ \tilde{\gamma})(t)$ seen as a function in $\R^n$.
%The second order Taylor expansion of $(\tilde{f}\circ \tilde{\gamma})(t)$ around $t=0$ reads
%\[ f(p) + \frac{\partial \tilde{f}}{\partial \tilde x^k}(\tilde{\gamma}(0)) \tilde{\gamma}'_k(0)\ t + \left[ \frac{\partial^2 \tilde{f}}{\partial \tilde x^k \partial \tilde x^h}(\tilde{\gamma}(0)) \tilde{\gamma}'_k(0)\tilde{\gamma}'_h(0) + \frac{\partial \tilde{f}}{\partial \tilde x^k}(\tilde{\gamma}(0)) \tilde{\gamma}''_k(0)\right] \ \frac{t^2}{2}.\] 
%We see that in this case the symbol of the two jet, namely the term between squared brackets, is precisely the diffusor in  \cref{eq:diffusor}. The first order part of the jet is the tangent vector. However the jet includes both the first and second order parts acting in their velocity and acceleration capacities, thus allowing us to see both of them. We know that the symbol of a two jet, which we have seen to coincide with the diffusor, is intrinsic, and so is the diffusor. However, the split of the diffusor or symbol in first order $\partial / \partial x^k$ and second order $\partial^2 / \partial x^k \partial x^h$ parts is not intrinsic. In fact one needs a connection to be able to define the first order part of a diffusor in an intrinsic way (see Emery \cite{emery}, p. 105).   

We now come to Schwartz morphism. Schwartz considers the unusual way to write the It\^o formula as an operator applied to a test function, and a diffusor more precisely. Let $Y_t$ be the solution to an SDE in the manifold $M$ in a coordinate system. Write 
\[ ``d^{Ito} f(Y_t) ='' (DY_t) f =  dY^i \frac{\partial}{\partial y^i} f(Y_t) +\frac{1}{2}  d[Y^i,Y^j] \frac{\partial^2}{\partial y^i \partial y^j} f(Y_t)\]
where $[\cdot,\cdot]$ is the quadratic covariation. 
Abstrating $DY_t$, we may define informally the diffusor 
\[  DY_t =  dY^i \frac{\partial}{\partial y^i} +\frac{1}{2}  d[Y^i,Y^j] \frac{\partial^2}{\partial y^i \partial y^j} . \]
This is indeed a second order differential operator without constant term. The definition is not fully rigorous since $d Y_i$ only makes sense in an integral expression, and indeed a slightly more precise version of Schwartz's principle is that forms of order 2 can be integrated along semimartingales, and that any object that can be integrated against all semimartingales is a second order form, or can be transformed in such a form (\cite{emery}, p. 80). 

We consider the special case of SDEs on a manifold driven by Brownian motion in Euclidean space $\R^d$. In this case we can write 
%
%We will have a similar definition for $DY$ in the manifold $N$, namely 
%\[  DY_t =  dY_i \frac{\partial}{\partial y^i} +\frac{1}{2}  d[Y_i,Y_j] \frac{\partial^2}{\partial y^i \partial y^j}. \]
%
%
a second order SDE as $D Y_t = F(Y_t) D W_t$, where $F(Y_t)$ is a smooth linear map that maps diffusors in the Euclidean space where $W$ evolves to diffusors on $M$. In coordinates, 
\[ dY^i \frac{\partial}{\partial y^i} +\frac{1}{2}  d[Y^i,Y^j] \frac{\partial^2}{\partial y^i \partial y^j} = F(Y_t)  \left( dW^i \frac{\partial}{\partial x^i} +\frac{1}{2}  d[W^i,W^j] \frac{\partial^2}{\partial x^i \partial x^j}\right) \]
where $x$ are orthogonal coordinates in the Euclidean space for $W$ and $y$ are local coordinates on the manifold. 
Since $F$ is linear
\[ dY^i \frac{\partial}{\partial y^i} +\frac{1}{2}  d[Y^i,Y^j] \frac{\partial^2}{\partial y^i \partial y^j} =  dW^i F(Y_t)\  \left( \frac{\partial}{\partial x^i}\right) +\frac{1}{2} d[W^i,W^j]\  F(Y_t)\left(  \frac{\partial^2}{\partial x^i \partial x^j}\right) \]
We can specify the way $F(Y)$ acts on $\partial / \partial x^i$ and 
$\partial^2 / \partial x^i \partial x^j$ in the given coordinate system at the given point $Y$. We will omit the specific point in part of the notation. 
Define
\[ F(Y_t)\  \left( \frac{\partial}{\partial x^i}\right) =F_i^{\alpha} \frac{\partial}{\partial y^\alpha }+  F_i^{\alpha,\beta} \frac{\partial^2}{\partial y^\alpha \partial y^\beta}\] 
 \[ F(Y_t)\  \left( \frac{\partial^2}{\partial x^i \partial x^j}\right) =F_{ij}^{\alpha} \frac{\partial}{\partial y^\alpha }+  F_{ij}^{\alpha,\beta} \frac{\partial^2}{\partial y^\alpha \partial y^\beta}.\] 
 Substituting and matching coefficients we obtain
% \[ dY_i \frac{\partial}{\partial y^i} +\frac{1}{2}  d[Y_i,Y_j] \frac{\partial^2}{\partial y^i \partial y^j} =  dX_i \left( F^i_{\alpha} \frac{\partial}{\partial y^\alpha }+  F^i_{\alpha,\beta} \frac{\partial^2}{\partial y^\alpha \partial y^\beta}\right) +\frac{1}{2} d[X_i,X_j]\left( F^{ij}_{\alpha} \frac{\partial}{\partial y^\alpha }+  F^{ij}_{\alpha,\beta} \frac{\partial^2}{\partial y^\alpha \partial y^\beta}\right) \]
%which, matching coefficients, leads to
\begin{eqnarray} \label{eq:quadraticcov1a}  && dY^i = dW^k F_k^i + \frac{1}{2} d[W^h,W^k] F_{hk}^i  \\  
\label{eq:quadraticcov1} &&   d[Y^i,Y^j] =2 F_k^{ij} dW^k + F_{hk}^{ij} d [W^h,W^k]. 
\end{eqnarray}
 The problem with this system is that it is partly redundant. The rules of stochastic calculus allow us to compute the quadratic covariation $d[Y^i,Y^j]$ directly from the $d Y^k$ equations, i.e.\ from the first of the last two equations we have that
 \begin{equation}\label{eq:quadraticcov} d[Y^i,Y^j] = F^i_h F^j_k d[X^h,X^k] 
\end{equation}
 so that we need to impose consistency between \eqref{eq:quadraticcov1} and   \eqref{eq:quadraticcov}: 
 \[ \mbox{(Schwartz morphism)} \ \ \  F_k^{ij} = 0, \ \ \ F_{hk}^{ij} = \frac{1}{2} [ F^i_h F^j_k  + F^j_h F^i_k]. \]
Any map $F$ satisfying this last equation above is called a ``Schwartz Morphism''. One can check that this is an intrinsic definition. 

Hence SDEs for $Y$ on manifolds $M$ driven by $W$ in Euclidean space can be written as a {\emph{subset}} of all smooth linear maps from $\R^d$ to $\Ti_M$, the subset being given by Schwartz morphisms. 

%One might have been tempted make the quadratic covariation explicit from the start, in coordinates, as
%\[ d [W^i, W^j] = dW^i d W^j = g_E^{ij} dt\]
%(Kronecker delta). However, this would be misleading. See the discussion in \cite{emery}, 6.21, p. 79. See also the discussion in this paper above \cref{sdeEinstein}. \rosso{(I don't understand
%this paragraph, the first sentence is ungrammatical. But perhaps it is unnecessary anyway?) }
 
Compute now the jets representation in coordinates. 
Go back to our notation
\[  P_t \jet j_2(\gamma_{P_t})(dW_t) \]
and let $\tilde{\gamma}$ denote also the same curve in $\R^n$ obtained through a chart with coordinates $y$. To contain notation we omit the tilde symbol on the curve, but it is understood we are working in coordinates. In coordinates $y$ we have that the 2-jet representation of the SDE reads
%, see Section \ref{sec:sdesmultiplebrownians},
\begin{equation}\label{eq:2jetdY}  d Y^i = \partial_{k}{\gamma^i_{Y_t}}(0) d W^k + \frac{1}{2}\partial^2_{k, h} \gamma^i_{Y_t}(0) [dW^k, dW^h]. 
\end{equation}

Comparing this last equation to eq. \eqref{eq:quadraticcov1a} we deduce that the two are equivalent, in the given coordinate sytem, if
%\[\partial_{k}{\gamma^i_{Y_t}}(0) = F^i_k, \ \ \partial^2_{h, k} \gamma^i_{Y_t}(0) = F^i_{h,k} .\]
%In fact the second condition is more than is needed. All one would need for the two equations to coincide is \rosso{Delete this sentence and the equation above to just give the correct formula?}
\[\partial_{k}{\gamma^i_{Y_t}}(0) = F^i_k, \ \ \partial^2_{k, k} \gamma^i_{Y_t}(0) = F^i_{k,k} .\]
The 2-jet and the diffusor/Schwartz morphism approaches are both intrinsic. This shows how
the Schwartz morphism can be computed from the $2$-jet. The Schwarz morphism has the advantage
that there is a unique Schwartz morphism associated with an SDE. The $2$-jet approach has the
advantage of using only familiar differential geometric concepts which are easy to interpret.

%Interestingly, we can view the map $v \mapsto j_2(\gamma_{P_t})(v)$ as a map from $\R^d$ to 2-jets of $\gamma$. 

\bigskip

A different approach to geometry of It\^o SDEs is the It\^o bundle approach (\cite{belopolskaja}, see also \cite{gliklikh} and the appendix in \cite{brzezniak}). If one wished to base our $2$-jet SDE construction on the It\^o bundle approach, following the notation in \cite{gliklikh}, one might proceed as follows. Having in mind \eqref{notEinstein}, rewritten informally as
\[ X_{t+dt} = X_t + B dW_t + \beta(dW_t,dW_t),\]
one could consider the map
\begin{equation}\label{eq:mapbundles} x +B z+\beta(z,z) \mapsto   \left(\frac{1}{2} \mbox{Tr}( \beta), B\right) 
\end{equation}
from sections of the 2-jet bundle  to sections of the It\^o bundle and argue that this is well defined, leading to a well defined SDE in the It\^o bundle sense. 

We will explore mappings between such bundles and the related SDEs equivalence more rigorously in future work. 
% end arxiv only
}{}

\section{Proof of convergence to classical It\^o calculus}
\label{proofsSection}

In this appendix we prove Theorem \ref{th:convergence}, stating convergence in $L^2(\Pb)$ of the 2-jet scheme to the classical It\^o solution of the SDE. 
The
techniques used to prove almost-sure convergence of the classical Euler scheme (\cite{gyongyas} for example) could be adapted to the $2$-jet scheme.

\begin{proof}
We think of $T$ as fixed and as $N$ increases ${\cal T}^N$ provides a finer discretization grid
approximating $[0,T]$.

To remove clutter from our equations, during this proof we will adopt the following conventions. $C$ is a constant, independent of $N$ that may change from line to line. We drop the superscript $N$ from terms such as $X_t^N$. Sums over Greek indices always range from $1$ to $d$. $i$, $j$ and $k$ are always
integers. Terms with integer time subscripts such as $X_i$ are shorthand for $X_{i \delta t}$. Superscript $T$'s indicate the vector transpose rather than the terminal time.

Under our hypotheses, we know from \cite{kloedenbook}, Theorem 10.2.2 p.\ 342 that the Euler scheme
\[
\bx_{t+\delta t} =  \bx_t + a( \bx_t) \delta t + \sum_{\alpha} b_\alpha( \bx_t) (W^\alpha_{t+\delta t} - W^\alpha_t) \ \mbox{for}\  \ t \in {\cal T}^N
\]
converges to the solution of the It\^o SDE in that
\begin{equation*}
\max_{t \in {\cal T}^N} E\left\{ |\bx_t - \tx_t|^2 \right\} \leq C\ \delta t \  .
\end{equation*}

Since
\[
\begin{split}
\max_{t \in {\cal T}^N} E\left\{ |X_t - \tx_t|^2 \right\} 
&\leq \max_{t \in {\cal T}^N} E\left\{ 2 |X_t - \bx_t|^2 + 2 |\bx_t - \tx_t|^2 \right\} \\
&\leq 2 \max_{t \in {\cal T}^N} E\left\{ |X_t - \bx_t|^2\right\} + 2\max_{t \in {\cal T}^N} E\left\{ |\bx_t - \tx_t|^2 \right\} \\
\end{split}
\]
all we need to conclude with \eqref{sup-msconv} for time points in ${\cal T}^N$ is to show
\[
\max_{t \in {\cal T}^N} E\left\{ |X_t - \bx_t|^2\right\} \leq C\ \delta t.
\]

Summing the differences of consecutive terms in the Euler scheme with time step $\delta t$ we have
\[
\bx_{k} = x_0 + \sum_{i=1}^k \left[ a( \bx_{i-1} ) \delta t +
\sum_\alpha b_\alpha( \bx_{i-1} ) \delta W^\alpha_{i-1} \right]
\]
where $\delta W_k:=W_{(k+1) \delta t}-W_{k \delta t}$. Using the definition for $a$ we can write
this as
\begin{equation}
\bx_{k} = x_0 + \sum_{i=1}^k \left[
 \sum_{\alpha,\beta} \tilde{a}_{\alpha,\beta}(\bx_{i-1}) g_E^{\alpha \beta} \delta t +
\sum_\alpha b_\alpha( \bx_{i-1} ) \delta W^\alpha_{i-1} \right]
\label{eulerSchemeAsSum}
\end{equation}
where $g_E$ is the metric tensor of the Euclidean metric and where we define
\[
\tilde{a}_{\alpha,\beta}(x) := \frac{1}{2} \frac{\partial^2 \gamma_x}{\partial u^\alpha \partial u^\beta} \Bigr|_{u=0}, \ \ \ \mbox{so that} \ \ \  a = \sum_{\alpha,\beta} \tilde{a}_{\alpha,\beta} g^{\alpha,\beta} .
\]
Note that we are not using the Einstein summation convention in this proof.
Let us write the $2$-jet scheme via its Taylor expansion as in \eqref{continuousTimeEquationW2withR}, namely

\begin{equation}
 X_{k} = X_{k-1}
 + \sum_{\alpha,\beta}\tilde{a}_{\alpha, \beta}(X_{k-1}) \delta W^\alpha_{k-1} \delta W^\beta_{k-1}+ \sum_{\alpha} b_\alpha({X_{k-1}}) \delta W^\alpha_{k-1}
 +  R_k.
\label{continuousTimeEquationW2withRproof}
\end{equation}
This expression defines the remainder $R_k$. Note that the remainder depends also on $N$ but our notation has suppressed this. Summing the differences of consecutive terms, and substituting in the definition of $b$ we obtain
\begin{equation}
X_{k} = x_0 + \sum_{i=1}^k \left[ 
 \sum_{\alpha,\beta} {\tilde{a}}_{\alpha,\beta}( X_{i-1} ) \delta W^\alpha_{i-1} \delta W^\beta_{i-1}
 + \sum_{\alpha} b_\alpha(X_{i-1}) \delta W^\alpha_{i-1}
 +  R_i
\right].
\label{curvedSchemeAsSum}
\end{equation}
Subtracting equation \eqref{eulerSchemeAsSum} from \eqref{curvedSchemeAsSum} we obtain
\[
X_k - \bx_k = P_k + \sum_\alpha Q^\alpha_k + \sum_{\alpha,\beta}S^{\alpha,\beta}_k 
\]
where we define
\[
P_k := \sum_{i=1}^k R_i, \qquad
Q^\alpha_k := \sum_{i=1}^k \left[ b(X_{i-1}) - b(\bx_{i-1}) \right] \delta W^\alpha_{i-1}
\]
\[
\begin{split}
\text{and }
S^{\alpha,\beta}_k &:= \sum_{i=1}^k \left[ 
{\tilde{a}}_{\alpha,\beta}(X_{i-1})
\delta W^\alpha_{i-1} \delta W^\beta_{i-1}
-
{\tilde{a}}_{\alpha,\beta}(\bx_{i-1})
g^{\alpha \beta}_E \delta t
\right].
\end{split}
\]
We have that
\[
E\left\{ | X_k - \bx_k|^2 \right\}
\leq C\left(
E\left\{ | P_k |^2 \right\} + \sum_\alpha E \left\{ | Q^\alpha_k|^2 \right\} + \sum_{\alpha,\beta} E \left\{ | S^{\alpha,\beta}_k  |^2 \right\}
\right) .
\]
We now wish to obtain bounds for each expectation on the right in terms of $\delta t$ and the function.
\[
Z(t) = \max_{0\leq s \leq t, s \in {\cal T}^N} E \left\{ |X_s - \bx_s|^2 \right\}.
\]

Using our bound on the third derivatives of $\gamma$, we can bound the remainder terms $R_i$
as follows
\[
|R_i| \leq C \Big( \sum_{\alpha} | \delta W^\alpha_{i-1} | \Big)^3.
\]
Writing $M:= E \Big\{ \Big(\sum_{\alpha} |\delta W^\alpha_{i-1}|\Big)^6 \Big\}$, we calculate that
\[
\begin{split}
E(|P_k|^2) = E \{ |\sum_i R_i|^2 \} &\leq C N^2  M \leq C \left(\frac{T}{\delta t}\right)^2 (\delta t)^\frac{6}{2} \leq C \delta t.
\end{split}
\]

By the discrete It\^o isometry and the Lipschitz condition on the derivatives we find 
\begin{equation*}
\begin{split}
E\Big\{ |Q^\alpha_k|^2 \Big\}
&= E\Big\{ \Big| \sum_{i=1}^k [ b_{\alpha}(X_{i-1}) - b_{\alpha}(\bx_{i-1}) ] \delta W^\alpha_{i-1} \Big|^2 \Big\} = E\Big\{ \Big| \sum_{i=1}^k [ b_{\alpha}(X_{i-1}) - b_{\alpha}(\bx_{i-1}) ]^2 \delta t  \Big| \Big\} \\
&\leq C E\Big\{ \Big| \sum_{i=1}^k | X_{i-1} - \bx_{i-1} |  \Big|^2  \delta t \Big\} 
\leq C \int_0^{k \delta t} Z(s) \, \ed s.
\end{split}
\end{equation*}

To bound $S^{\alpha, \beta}$ we write it as a sum of two components $S^{\alpha,\beta,1}$ and $S^{\alpha,\beta,2}$ defined as follows.
\[
\begin{split}
S^{\alpha,\beta,1}_k &:= \sum_{i=1}^k 
{\tilde{a}}_{\alpha,\beta}(\bx_{i-1}) (\delta W^\alpha_{i-1} \delta W^\beta_{i-1} - g_E^{\alpha,\beta} \delta t ) \\
S^{\alpha,\beta,2}_k &:= \sum_{i=1}^k [{\tilde{a}}_{\alpha,\beta}(X_{i-1}) - {\tilde{a}}_{\alpha,\beta}(\bx_{i-1})]
\delta W^\alpha_{i-1} \delta W^\beta_{i-1} \\
\end{split}
\]
By expanding the square term using the definition of $S^{\alpha,\beta,1}$ we write $E\left\{ |S^{\alpha,\beta,1}_k|^2 \right\}$ as
\begin{equation*}
\sum_{i=1}^k \sum_{j=1}^k E \left\{
{\tilde{a}}_{\alpha,\beta}^T(\bx_{i-1}) {\tilde{a}}_{\alpha,\beta}(\bx_{j-1})
( \delta W^\alpha_{i-1} \delta W^\beta_{i-1} - g^{\alpha,\beta}_E \delta t )
( \delta W^\alpha_{j-1} \delta W^\beta_{j-1} - g^{\alpha,\beta}_E \delta t )
\right\}.
\end{equation*}
When $i\neq j$ the terms on the right hand side vanish. This is because we may assume
WLOG that $j>i$ in which case the last factor of the $(i,j)$-th term,
\[
\delta W^\alpha_{j-1} \delta W^\beta_{j-1} - g^{\alpha,\beta}_E \delta t
\]
is independent of the rest of the term and has expectation $0$. We now quote 10.2.14 p.\ 343 in
\cite{kloedenbook} to show
\begin{equation}
E\left\{ |\bx_{k}|^2 \right\} \leq C.
\label{l2Bound}
\end{equation}
We now compute
\begin{equation*}
\begin{split}
E\left\{ |S^{\alpha,\beta,1}_k|^2 \right\} &=
\sum_{i=1}^k E \left\{
|{\tilde{a}}_{\alpha,\beta}(\bx_{i-1})|^2
( \delta W^\alpha_{i-1} \delta W^\beta_{i-1} - g^{\alpha,\beta}_E \delta t )^2
\right\} \\
&\leq \sum_{i=1}^k E \left\{ | {\tilde{a}}_{\alpha,\beta}(\bx_{i-1})|^2  \right\}
E\left\{ (\delta W^\alpha_{i-1} \delta W^\beta_{i-1} - g^{\alpha,\beta}_E \delta t)^2 \right\} \\
&\leq \sum_{i=1}^k C (\delta t)^2 E \left\{ | {\tilde{a}}_{\alpha,\beta}(\bx_{i-1})|^2  \right\} \\
&\leq \sum_{i=1}^k C (\delta t)^2 E \left\{ (1 + \bx_{i-1}^2) \right\}
\leq \sum_{i=1}^k C (\delta t)^2
\leq C \delta t.
\end{split}
\end{equation*}
The second line follows from independence of Brownian motion increments. The third from the second by the scaling properties of Brownian motion increments. The fourth from our growth bounds on the second derivatives. We then use the bound \eqref{l2Bound}.  
Let us write $\Delta_{\alpha,\beta,i}:=\tilde{a}_{\alpha,\beta}(X_i) - \tilde{a}_{\alpha,\beta}(\bx_i)$.
We find
\begin{equation*}
\begin{split}
& E \left \{ | S^{\alpha,\beta,2}_k |^2 \right\} \\
&= 2 \sum_{i=0}^{k-1} \sum_{j=0}^i E \left\{ \Delta_{\alpha,\beta,i}^T
\Delta_{\alpha,\beta,j}
 \delta W^\alpha_i \delta W^\beta_i
\delta W^\alpha_j \delta W^\beta_j  \right \} \\
&= 2 \sum_{i=0}^{k-1} \left( E \left\{ |\Delta_{\alpha,\beta,i}|^2 \right\} E\left\{ (\delta W^\alpha_i \delta W^\beta_i)^2 \right\} + \sum_{j=0}^{i-1} E \left\{  \Delta_{\alpha,\beta,i}^T
\Delta_{\alpha,\beta,j}
\delta W^\alpha_j \delta W^\beta_j \right\} E\left\{ \delta W^\alpha_i \delta W^\beta_i \right \} 
\right)
\\
&\leq C \sum_{i=0}^{k-1} \left( (\delta t)^2 E \left\{ |\Delta_{\alpha,\beta,i}|^2 \right\} +   \delta t \sum_{j=0}^{i-1} E \left\{  \Delta_{\alpha,\beta,i}^T
\Delta_{\alpha,\beta,j}
\delta W^\alpha_j \delta W^\beta_j \right\} \delta_{\alpha,\beta}
\right)
\\
&\leq C \sum_{i=0}^{k-1} \left( (\delta t)^2 E \left\{ |\Delta_{\alpha,\beta,i}|^2 \right\} + \delta t \sum_{j=0}^{i-1} E \left\{| \Delta_{\alpha,\beta,i}^T|^2 \right\}^\frac{1}{2} E\left\{ |\Delta_{\alpha,\beta,j}^T
\delta W^\alpha_j \delta W^\beta_j |^2 \right\}^\frac{1}{2} \right) \\
&\leq C \sum_{i=0}^{k-1} \left( (\delta t)^2 E \left\{ |\Delta_{\alpha,\beta,i}|^2 \right\} +
\delta t \sum_{j=0}^i E \left\{| \Delta_{\alpha,\beta,i}|^2 \right\}^\frac{1}{2} E\left\{ |\Delta_{\alpha,\beta,j}|^2  \right\}^\frac{1}{2} E \left\{
|\delta W^\alpha_j \delta W^\beta_j |^2 \right\}^\frac{1}{2} \right) \\
&\leq C (\delta t)^2 \sum_{i=0}^{k-1} \sum_{j=0}^i
\left(
E \left\{ |\Delta_{\alpha,\beta,i}|^2 \right \} 
+ E \left\{ |\Delta_{\alpha,\beta,j}|^2 \right \}
\right).  \\
\end{split}
\end{equation*}
We have used: independence; the scaling properties of Brownian increments; $E\left\{ \delta W^\alpha_i \delta W^\beta_i \right \} = \delta_{\alpha,\beta} \delta t$ (Kronecker delta); Cauchy-Schwarz; independence; the scaling of Brownian increments and Young's Inequality.
Applying now
the Lipschitz property of ${\tilde{a}}_{\alpha,\beta}$ we find
\begin{equation*}
\begin{split}
E \left \{ | S^{\alpha,\beta,2}_k |^2 \right\}
&\leq C (\delta t)^2 \sum_{i=0}^{k-1} \sum_{j=0}^i
\left(
E \left\{ |X_i - \bx_i|^2 \right \}
+ E \left\{ |X_j - \bx_j|^2 \right \}
\right)   \\
&\leq C \delta t \sum_{i=0}^{k-1} Z( i \delta t) \leq C \int_0^{k \delta t} Z(t) \ dt
\end{split}
\end{equation*}
where we took into account the fact tha
$i \le N = T/\delta t$, so that $i (\delta t)^2 \le C \delta t$.
Putting together all our bounds we have
\[ 
Z(t) \leq C\left( \delta t + \int_0^t Z(s) \, \ed s  \right).
\]
So by the Gronwall inequality, $Z(t) \leq C \delta t$.

We have established that
\begin{equation}
\label{discreteTimeEstimate}
\max_{t \in {\cal T}^N} E\left\{ |X^{N}_t - \tx_t|^2 \right\} \le C\ \delta t.
\end{equation}
To complete the proof of \eqref{sup-msconv}, let $t \in [0,T]$ be a time point not necessarily in the grid. Using Theorem 4.5.4 of \cite{kloedenbook} we can obtain a bound $E(|\tilde{X}_{t_i}|^2)\leq C$ and hence a bound
$E(|X_{t_i}|^2) \leq C$.
Applying Theorem 4.5.4 of \cite{kloedenbook} a second time we have
\[
E(|\tilde{X}_{t}-\tilde{X}_{t_i}|^2) \leq C \left(1 + E(|\tilde{X}_{t_i}|^2)\right) \delta t \leq C \delta t.
\]
By the definition of our scheme, our estimates on the derivatives of $\gamma$ and Taylor's theorem
\[
E(|X_{t}-X_{t_i}|^2) \leq C \left( 1 + E(|X_{t_i}|^2)\right) \delta t \leq C \delta t.
\]
The inequality \eqref{sup-msconv} now follows.

\end{proof}

\section{A coordinate free notion of convergence}
\label{coordinateFreeConvergence}

The notion of mean square convergence depends upon one's choice of coordinates and so is not the best notion
of convergence to use in differential geometry. Let us briefly describe an alternative form of convergence which can be used instead.
\ifarxiv{}{See Appendix D of \cite{arxivVersion} for full details.}

Let $M$ be a manifold and $g$ be a Riemannian metric on $M$. Let $K$ be a compact subset $M$. Let $d^g$ denote the Riemannian distance function. Let $K^0$ denote the interior of $K$. We define an equivalence relation $\sim$ 
on $M$ by $x \sim y$ if either $x=y$ or both $x \notin K^0$ and $y \notin K^0$.
The quotient space $M/\sim$ is simply the one-point compactification of $K^0$. We write $\infty$ for the equivalence class consisting of all points outside $K^0$. 

We may define a semimetric $\tilde{d}^{g,K}$ on $M/\sim$ by
\[
\tilde{d}^{g,K}([x],[y])=\underset{X\sim x,Y \sim y}{\inf}d^g(X,Y).
\]
This is not a metric since $\tilde{d}^{g,K}$ does not obey the triangle inequality. Nevertheless, convergence of a sequence in $\tilde{d}^{g,K}$ implies convergence in
$M/\sim$\ifarxiv{\ (see Lemma \ref{lemma:quotientMetric}).}{.}

Given a stochastic process $X:[0,T] \to M \cup \{ \infty \}$ and a compact subset $K$ of $M$ we define a new stochastic process $X^K$ by
\[
X^K_t(\omega) = \begin{cases}
X_t(\omega) & \text{if }X_{t^\prime}(\omega) \in K^0 \text{ for all }t^\prime< t \\
\infty & \text{otherwise}.
\end{cases}.
\]

\begin{definition}
Let $X^i$ be a sequence of stochastic processes in $M \cup \{\infty\}$. For a fixed time $t$,
we say that $X^i$ converges to $X$ {\em in mean square on compacts} if for all $\epsilon>0$, compact sets $K \subseteq M$
and Riemannian metrics $g$ on $M$ there exists $N \in \mathbb{N}$ such that if $i \geq N$
\[
E(\tilde{d}^{g,K}((X^i)^K_t(\omega), X^K_t(\omega))^2 ) \leq \epsilon
\]
\end{definition}

\ifarxiv{This definition is stated in terms of all Riemannian metrics on $M$ so it is manifestly coordinate free.}{}

If the coefficients of an SDE are smooth then all the bounds required in the proof of Theorem \ref{th:convergence} will automatically hold over any compact set. As shown in \cite{hsu} any smooth Ito SDE on a manifold has a unique solution in $M \cup \{\infty\}$ if one sets the value of the solution to $\infty$ at the explosion time. Our $2$-jet scheme will converge in mean square on compacts to this solution of the corresponding It\^o SDE\ifarxiv{\ (see Theorem \ref{th:meanSquareOnCompacts}).}{.}

\ifarxiv{
	
Let us fill in the missing details. We need to prove that convergence in $\tilde{d}^{g,K}$ implies convergence in the quotient topology on $M/\sim$. We also need to show how Theorem \ref{th:convergence} implies convergence in mean squared on compacts for smooth SDEs.

Let us begin by considering the convergence on $M/\sim$. We recalling the definition of the quotient of a metric.
\begin{definition}[Quotient of a metric]
Let $(X,d)$ be a metric space and let $\sim$ be an equivalence relation. We define a bilinear map, also denoted $d$, on the quotient by
\[
d([x],[y])=\inf\{ d(p_1,q_1)+d(p_2,q_2)+\ldots +d(p_n,q_n) \}
\]
where the infimum is taken over finite sequences of pairs of points $(p_i,q_i)$ satisifying 
\begin{equation}
[x]=[p_1], \quad  y=[q_n] \quad \text{ and } \quad [q_i]=[p_{i+1}].
\label{eq:chain}
\end{equation}
Here $[x]$ denotes the equivalence class of $x$. $d$ is called the quotient pseudometric as in general $d([x],[y])=0$ may not imply $[x]=[y]$.
\end{definition}

Let us write $d^{g,K}$ for the quotient pseudometric of $d^g$ on $M/\sim$. It is immediate from our definitions that $d^{g,K}([x],[y])\leq \tilde{d}^{g,K}([x],[y])$. So convergence in $\tilde{d}^{g,K}$ implies convergence in $d^{g,K}$. That this then implies convergence in $M/\sim$ follows from the following lemma.

\begin{lemma}
$d^{g,K}$ is a metric that induces the quotient topology on $M/\sim$. Moreover this quotient topology is the equal to the one-point compactification topology on $K^0$.
\label{lemma:quotientMetric}
\end{lemma}
\begin{proof}
We'll call a sequence of points $(p_i,q_i)$ satisfying \eqref{eq:chain} a {\em chain}. We will call the sum of $d(p_i,q_i)$ the length of the chain. We will say that a chain can be shortened if there is another chain joining $x$ and $y$ with a lower value for $n$ and with a smaller length. The infimum can be obtained by taking the infimum over chains which cannot be shortened.

Let us show that the maximum length of a chain that cannot be shortened is $2$. Suppose that one for one $q_i$ with $i<n$, we have $q_i \in K^0$. Then $p_{i+1}=q_i$. So $d(p_i,q_i)+d(p_{i+1},q_{i+q})=d(p_i,q_i)+d(q_i,q_{i+1})\geq d(p_i,q_{i+1})$. So if a chain cannot be shortened $q_i$ cannot be in $K^0$ for any $i<n$. We deduce that a chain that cannot be shortened must have $p_{i+1} \notin K^0$ for any $i<n$. So either ${i+1}=n$ or else we could simply drop the term $d(p_{i+1},q_{i+1})$ to shorten our chain. Thus $n\leq 2$. We have shown that
\[
d^{g,K}([x],[y])=\min\{ d^g(x,y), d^g(x,M\setminus K^0) + d^g(y,M\setminus K^0)\}
\]
where $d^g(x,A)$ denotes the distance between a point $x$ and a set $A$.
To show that $d^{g,K}$ is a metric we just need to show that $d^{g,K}([x],[y])$ implies $[x]=[y]$. Since $K^0$ is open, if $x \in K^0$, $d^g(x, M\setminus K^0)>0$. So $d^{g,K}(x,y)=0$ if and only if either $x=y$ or both $x \notin K^0$ and $y \notin K^0$.

By definition of the quotient topology, a set $U \in M/\sim$ is open if and only if $\pi^{-1}U$ is open in $M$ where $\pi$ is the projection map. 

To show that the metric induces the quotient topology, we will show
\begin{enumerate}[(i)]
	\item that the $\epsilon$-balls of the metric are open sets in the quotient topology
	\item given any open set $U$ in the quotient topology and any point $[x] \in U$ we can find an $\epsilon$-ball around $x$ that is contained in $U$.
\end{enumerate}

Let us begin with item (i).

Write $\infty$ for the equivalence class of any $y \notin K^0$.
Let $B_\epsilon^{d^{g,K}}(\infty)=\{[x]:d^{g,K}([x],[\infty])<\epsilon\}$.
So
\[
\pi^{-1}B_\epsilon^{d^{g,K}}(\infty)=\{x:d^g(x,M\setminus K^0) < \epsilon \}
= \cup_{y \in M\setminus K^0}\{x:d^g(x,y)< \epsilon \}.
\]
This is an open set since uncountable unions of open sets are open.
Hence $B_\epsilon^{d^{g,K}}(\infty)$ is open in the quotient topology.

Let $x$ be a point in $K^0$ and let $\epsilon_1=d^g(x,M\setminus K^0)$.
If $\epsilon \leq \epsilon_1$
\[
\pi^{-1}B_\epsilon^{d^{g,K}}([x])=B_\epsilon^{d^g}(x),
\]
otherwise
\[
\pi^{-1}B_\epsilon^{d^{g,K}}([x])=B_\epsilon^{d^g}(x) \cup \pi^{-1} B_{\epsilon-\epsilon_1}^{d^{g,K}}(\infty).
\]	
In either case we see that $\pi^{-1}B_\epsilon^{d,g}([x])$ is open.
Hence $B_\epsilon^{d^{g,K}}([x])$ is open in the quotient topology.

Now let us prove item (ii).

Let us first consider the case when $x \in K^0$. Since $K^0$ is open, we can find an $\epsilon_1$ such that $B_{\epsilon_1}^{d^g}(x) \subseteq K^0$. Since $\pi^{-1}U$ is open, we can find an $\epsilon_2$ such that  $B_{\epsilon_2}^{d^g}(x) \subseteq \pi^{-1}U$. Let $\epsilon=\min\{\epsilon_1, \epsilon_2\}$ then $B_\epsilon^{d^g}(x) \subseteq K^0 \cap \pi^{-1}U$. So $B_\epsilon^{d^{g,K}}([x]) = \pi( B_\epsilon^{d^g}(x)) \subseteq U$.

We now suppose $x \notin K^0$, so $[x]=\infty$. We need to show that for some $\epsilon>0$, $B^{d^{g,K}}_\epsilon(\infty) \subseteq U$. Equivalently we must show that for some $\epsilon>0$, $d^g(x, M\setminus K^0)<\epsilon \implies x \in \pi^{-1}{U}$. Equivalently, we must show that for some $\epsilon>0$, $x \in M\setminus \pi^{-1}U \implies d^g(x, M\setminus K^0)\geq \epsilon$.

We suppose for a contradiction that for all $\epsilon>0$ there exists $x \in M \setminus \pi^{-1}(U)$ with $d^g(x,M \setminus K^0) < \epsilon$. So we can define a sequence $x_n$ of points in $M\setminus \pi^{-1}(U)$ with $d^g(x,M \setminus K^0) < \frac{1}{n}$. Since $x_n \in M\setminus \pi^{-1}(U)$ and $\infty \in U$  we see that $x_n \in K^0$. $K$ is compact, so $x_n$ has a convergent subsequence. We may therefore assume without loss of generality that $x_n$ converges to some $x_\infty$. Given any $e>0$ we can find an $N$ such that for all $n\geq N$, $d(x_n, x_\infty)<e$ and hence $d(x\infty, M \setminus K^0) < e + \frac{1}{n}$. We deduce that $d(x_\infty, M\setminus K^0)=0$. Since $M\setminus K^0$ is closed, $x_\infty \in M\setminus K^0$. So $\pi(x_\infty)=\infty$ and so $x_\infty \in \pi^{-1}U$. Since $\pi^{-1}U$ is open, for some $e>0$, $B_e(x\infty) \subseteq \pi^{-1}U$. Now, for large enough $n$, $x_n \in B_e(x_\infty) \subseteq \pi^{-1}U$. This contradicts the defining property of $x_n$ that $x_n \in M \setminus \pi^{-1}(U)$.

At this stage we've proved that $d^{g,K}$ is a metric and it coincides with the quotient topology. We've also asserted that this quotient topology is equal to the one-point compactification topology. Let's prove that.

In the quotient topology, the open sets which contain $\infty$ are the image under $\pi$ of open sets which contain $M \setminus K^0$. That's the same thing as the complement of the closed sets which don't intersect $M \setminus K^0$. This is the same thing as the complement of the closed sets in $K^0$. Since any closed set of $K^0$ will be a closed subset of $K$ it will be compact. So the open sets containing $\infty$ are the image under $\pi$ of the complement of the compact sets in $K^0$. This now corresponds to the usual definition of the open sets in the one-point compactification topology.
\end{proof}

We now wish to show that Theorem \ref{th:convergence} implies convergence in mean square on compacts for SDEs with smooth coefficients. As a preliminary step, we prove some basic results on the equivalence of metrics on compact manifolds that are well-known but whose proofs we found hard to track down.

We recall the definition of strong equivalence of metrics.
\begin{definition}
Two metrics $d^1$ and $d^2$ are said to be strongly equivalent if there exist constants $\lambda>0$ and $\mu>0$ such that
\[
\lambda d^1(x,y) \leq d^2(x,y) \leq \mu d^1(x,y)
\]
\end{definition}
We make a new definition:
\begin{definition}
Two metrics $d^1$, $d^2$ for a topological space $X$ are said to be locally strongly equivalent if for each $x \in X$ there is a neighbourhood of $X$ over which $d^1$ and $d^2$ are strongly equivalent.
\end{definition}
For compact manifolds these definitions coincide.
\begin{lemma}
If $d^1$ and $d^2$ are locally strongly equivalent metrics for a compact topological space $X$ then $d^1$ and $d^2$ are strongly equivalent.
\end{lemma}
\begin{proof}
It suffices to prove that there exists $\mu>0$ such that $d^1(x,y) \leq \mu d^2(x,y)$.	 For we may then prove that there exists $\lambda^\prime>0$ such that $d^1(x,y)\leq \lambda^\prime d^2(x,y)$ which we rearrange to obtain $\frac{1}{\lambda^\prime}d^1(x,y) \leq d^2(x,y)$. Take $\lambda=\frac{1}{\lambda^\prime}$.
	
So suppose for a contradiction that for all $\mu$ we can find $x$ and $y$ 
with $d^2(x,y) > \mu d^1(x,y)$. We can then choose sequences $x_n$ and $y_n$ with $d^2(x_n,y_n) > n d^1(x_n,y_n)$. By compactness of $X$ we may assume without loss of generality that $x_n$ converges to some $x_\infty$ as $n\to \infty$ and that $y_n$ converges to $y_\infty$. Since $X$ is compact, $d^2(x,y)$ is bounded above by some finite diameter $D$. So $d^1(x_n,y_n)<\frac{D}{n}$ which tends to $0$ as $n\to \infty$. We deduce that $x_\infty=y_\infty$.

We now choose a neighbourhood $U$ of $x_\infty$ on which $d^1$ and $d^2$ are strongly equivalent. So there exists some $\mu^\prime>0$ such that
\[
d^2(x,y) < \mu^\prime d^1(x,y) 
\]
whenever $x$ and $y$ lie in $U$. For large enough $n$, both $x_n$ and $y_n$ will lie in $U$. So $d^2(x_n,y_n) < \mu^\prime d^1(x_n,y_n)$. Taking $n>\mu^\prime$, this contradicts the fact that $d^2(x_n,y_n) > n d^1(x_n,y_n)$.
\end{proof}
\begin{corollary}
Let $M$ be a compact manifold embedded in $\R^P$. Let $d_E$ denote the metric on $M$ induced by the Euclidean metric on $\R^P$. Let $g_E$ denote the Riemannian metric induced on $M$. Then $d_E$ and $d^{g_E}$ are equivalent.
\end{corollary}
\begin{proof}
Using Taylor's theorem, one can show that the $d_E$ and $d^{g_E}$ are locally strongly equivalent. This is simply a formalization of the argument that $d_E$ is a first approximation to $d^{g_E}$ with Taylor's theorem providing precise estimates.
\end{proof}
\begin{corollary}
Let $M$ be a compact manifold, let $g_1$ and $g_2$ be Riemannian metrics
on $M$ then $d^{g_1}$ and $d^{g_2}$ are equivalent.
\end{corollary}
\begin{proof}
Another application of Taylor's theorem.
\end{proof}

\begin{theorem}
Let $M$ be a finite dimensional manifold and let $\gamma_x$ be a smoothly varying field of $2$-jets defining an SDE on $M$. Then the scheme \eqref{eq:schemediscretemultiW} converges in mean-square on compacts to the the classical It\^o solution of the corresponding SDE \eqref{eq:equivalentClassical}.
\label{th:meanSquareOnCompacts}
\end{theorem}
\begin{proof}
By the Whitney embedding theorem, we can find an embedding of our manifold in $\R^r$ for large enough $r$. We may then smoothly extend an SDE defined on $K$ to give an SDE defined on the whole of $\R^r$ and which is trivial ouside of some compact subset $K^\prime$ in $\R^r$. Convergence of our scheme for this SDE in the Euclidean metric now follows by Theorem \ref{th:convergence}. Since $K$ is compact, the Euclidean metric on $K$ is strongly equivalent to $d^{g_E}$ where $g_E$ is the Riemannian metric induced on $M$ by the Euclidean metric. This in turn is strongly equivalent on $K$ to $d^g$ for any Riemannian metric $g$ on $M$. Let us write $X_t$ for the solution of the It\^o SDE and $X_t^i$ for convergent sequence of approximations obtained using our curved scheme. We see that there exists $N \in \mathbb{N}$ such that if $i \geq N$
\[
E(\tilde{d}^{g,K}((X^i)^K_t(\omega), X^K_t(\omega))^2 ) \leq \epsilon.
\]		
\end{proof}

% End of arxiv only content
}{}

\end{appendix}
\renewcommand{\newblock}{}

\bibliographystyle{plain}

% John had the line below. I have removed the ``../''
%\bibliography{../drawingsdes}
\bibliography{drawingsdes}

\end{document}

\end{document}